\documentclass[12pt,a4paper]{amsart}

\usepackage{amsmath,amssymb,latexsym,tikz,multicol}
\usepackage{arydshln,multirow,a4wide}

\newtheorem{theorem}{Theorem} [section]
\newtheorem{lemma}[theorem]{Lemma}
\newtheorem{corollary}[theorem]{Corollary}
\newtheorem{proposition}[theorem]{Proposition}

\newtheorem{example}[theorem]{Example}
\newtheorem{remark}[theorem]{Remark}

\usepackage{tikz}
\usepackage{enumerate}
\usetikzlibrary{matrix,arrows}
\usetikzlibrary{positioning}
\usetikzlibrary{patterns}

\usepackage{enumerate}
\numberwithin{equation}{section}

\subjclass[2010]{06E75 (primary), 08A05, 08B20 (secondary)} 

\title{Free skew Boolean algebras}

\author{Ganna Kudryavtseva}
\address{G. Kudryavtseva: 
Faculty of Civil and Geodetic Engineering, University of Ljubljana, Jamova cesta~2, SI-1000 Ljubljana, Slovenia}

\email{ganna.kudryavtseva\symbol{64}fgg.uni-lj.si}

\author{Jonathan Leech}
\address{J. ~Leech: Department of Mathematics,
Westmont College, 955 La Paz Road,
Santa Barbara, CA 93108,
USA}
\email{leech\symbol{64}westmont.edu}

\thanks{G. Kudryavtseva was partially supported by  ARRS grant P1-0288 and by Jo\v{z}ef Stefan Institute (IJS)}

\begin{document}

\maketitle

\begin{abstract} We study the structure and properties of free skew Boolean algebras. For finite generating sets, these free algebras are finite and we give their representation as a product of primitive algebras and provide formulas for calculating their cardinality. We also characterize atomic elements and central elements and calculate the number of such elements. These results are used to study minimal generating sets of finite skew Boolean algebras. We also prove that the center of the free infinitely generated algebra is trivial and show that all free algebras have intersections.

\vspace{0.5cm}

\end{abstract}

\section{Introduction}\label{s0:introduction}

Skew Boolean algebras (abbreviated SBAs) are non-commutative variants of generalized Boolean algebras (abbreviated GBAs). They are algebras $(S;\wedge,\vee,\setminus, 0)$ of signature $(2, 2, 2, 0)$ satisfying the usual axioms for generalized Boolean algebras, except for the commutativity of the operations $\wedge$ and $\vee$. (See Section \ref{s1:background} for details and further definitions.)

Skew Boolean algebras in some form were studied first in Australia by William Cornish and his student, Robert Bignall, and then in the United States by Jonathan Leech. (See~\cite{B,BL,C,L2}.)
They arise as algebraic structures defined on sets of partial functions, much as the subsets of a given set $X$ form a Boolean algebra under certain well-known operations. They also occur in rings where skew Boolean operations can be defined on various subsets of idempotents and in particular in rings whose full set of idempotents is closed under multiplication, much as any maximal set of commuting idempotents in a ring forms a GBA. (See \cite{CV2,CVL1,CVL2}.)  Discriminator varieties, an area of interest in universal algebra, have strong connections to SBAs that have intersections. (See  \cite[IV.9]{BS} and \cite{BL}. In this paper, the term intersection is
introduced in the final section.) Connections with logic and computer science have been addressed in \cite{CVLS,SBV,VS}. Various algebraic structures with close ties to skew Boolean algebras have been recently studied  in \cite{BJSV,Cir,CVS}. Stone duality has been extended from Boolean algebras to skew Boolean algebras in \cite{BCV,Kud,Kud1,KudLaw}. It is remarkable, however, that thus far free SBAs have received no systematic study. Thus, in this bicentennial year of George Boole (as of this writing), we seek to rectify this.

The purpose of this paper is to study free skew Boolean algebras. It begins by reviewing basic concepts and results  in Section \ref{s1:background}. Since finitely generated free SBAs are finite, a closer look at finite algebras occurs in Section \ref{s2:primitive}. The key concepts in this section are those of the orthogonality of elements or of primitive subalgebras of a given SBA, which leads to the important notion of an orthosum of primitive algebras. Results in this section are used in Section \ref{s3:finite} to analyze finitely generated free SBAs. The main results thereof are Theorems \ref{th:3.2} and \ref{th:3.3} that characterize the structure of term algebras on finitely many generators - term algebras being the default versions of free algebras. Applications to finite SBAs in general occur in Section \ref{s4:free_min_gen_set}, which focuses on determining of the size of minimal generating sets of finite SBAs (the inverse optimization problem to that solved by free algebras: maximizing the algebra, given a fixed set of generators). In the final Section \ref{s5:atom_splitting} we consider arbitrary (possibly infinite) free algebras. We prove results about the center (Theorem \ref{th:5.2}) and about intersections (Theorem \ref{th:5.3} and Corollary \ref{cor:5.4}). Multiple characterizations of free skew Boolean algebras are given in Theorem \ref{th:5.5}, which in turn is used in Theorem \ref{th:5.8} to verify an alternative construction of free left-handed SBAs - alternative, of course, to term algebras.

Finally, although a good bit of background is given in Section \ref{s1:background}, more detailed background is given in \cite{BL,L2,L6} and the early survey article, \cite{L3}.  It should also be
mentioned that, unless stated otherwise, our universal algebraic terminology is consistent
with that found in Burris and Sankappanavar \cite{BS}.

\section{Background} \label{s1:background}

A {\em skew lattice} is an algebra ${\bf S} = (S; \wedge, \vee)$ where $\wedge$ and $\vee$ are associative binary operations on a set $S$ that satisfy the absorption identities:
$$
x\wedge (x\vee y) = x = (y\vee x)\wedge x \, \text{ and } \, x\vee (x\wedge y) = x = (y\wedge x)\vee x.
$$
Both operations are necessarily idempotent and the following dualities hold:
$u\wedge v=u \Leftrightarrow u\vee v=v$ and $u\wedge v=v \Leftrightarrow u\vee v=u$.
The reducts $(S;\vee)$ and $(S;\wedge)$ are {\em regular} bands, that is, semigroups of idempotents satisfying $xyxzx = xyzx$. All skew lattices possess a coherent {\em natural partial order}: $x\geq y$ if $x\wedge y =y = y\wedge x$ or dually $x\vee y = x = y\vee x$. This refines the {\em natural preorder}: $x\succeq y$  if $y\wedge x\wedge y =y$  or dually $x\vee y\vee x =x$.

A family of equivalences, ${\mathcal D}$, ${\mathcal L}$ and ${\mathcal R}$, known as the {\em Green's relations} in semigroup theory are relevant also to skew lattices. Given a skew lattice ${\bf S}$, an equivalence relation ${\mathcal D}$ is defined on $S$ via the natural preorder by $x\mathrel{\mathcal{D}} y$ if $x\succeq y\succeq x$.  Thus $x\mathrel{\mathcal{D}} y$ iff $x\wedge y\wedge x=x$ and $y\wedge x\wedge y=y$, or, equivalently, $x\vee y\vee x=x$ and $y\vee x\vee y=y$.  The {\em Clifford-McLean Theorem} for bands (semigroups of idempotents) extends to skew lattices. Thus: 
\begin{enumerate}
  \item ${\mathcal D}$ is congruence on ${\bf S}$;
  \item ${\bf S}/{\mathcal D}$ is the maximal lattice image of {\bf S};
  \item  each ${\mathcal D}$-class is a maximal rectangular subalgebra of ${\bf S}$. 
\end{enumerate}
That is, if $D$ is a ${\mathcal D}$-class, then $(D, \wedge)$ and $(D, \vee)$ are rectangular bands (satisfying the identity $xyx = x$). They also jointly satisfy $x\wedge y=y\vee x$.  Finally, each ${\mathcal D}$-class is {\em anticommutative}: $x\wedge y=y\wedge x$ (or $x\vee y=y\vee x$) iff $x=y$. (See \cite{L3}.)

The congruence ${\mathcal D}$ is refined by a pair of congruences, ${\mathcal L}$ and ${\mathcal R}$, given by
$x\mathrel{\mathcal{L}} y$ if $x\wedge y=x$ and $y\wedge x=y$, or, equivalently, $x\vee y=y$ and $y\vee x=x$.
Likewise, $x\mathrel{\mathcal{R}} y$ if $x\wedge y=y$ and $y\wedge x=x$, or, equivalently, $x\vee y=x$ and $y\vee x=y$. ${\mathcal L}\cap {\mathcal R}=\Delta$,
the identity equivalence, while under the composition of relations, ${\mathcal L}\circ {\mathcal R}={\mathcal R}\circ {\mathcal L} = {\mathcal L}\vee{\mathcal R}={\mathcal D}$.
A skew lattice ${\bf S}$ is {\em left-handed} ({\em right-handed}) if and only if ${\mathcal D} = {\mathcal L}$ (${\mathcal D} = {\mathcal R}$). Equivalently, ${\bf S}$ is left-handed if both
$x\wedge y\wedge x=x\wedge y$
and dually $x\vee y\vee x=y\vee x$
hold ($x\wedge y\wedge x=y\wedge x$
and  $x\vee y\vee x=x\vee y$ hold). For any skew lattice ${\bf S}$, the canonical
maps ${\bf S}\to{\bf S}/{\mathcal R}$, ${\bf S}\to{\bf S}/{\mathcal L}$ and ${\bf S}\to{\bf S}/{\mathcal D}$  are universal homomorphisms from ${\bf S}$ to the
respective varieties of left- and right-handed skew lattices and lattices. This induces the pullback diagram

\begin{center}
\begin{equation}\label{eq:diagram}
\begin{tikzpicture}[baseline=(current  bounding  box.center)]
\node (a) {${\bf S}$};
\node(b) [node distance=2 cm, right of=a] {${\bf S}/{\mathcal R}$};
\node (a') [node distance=1.6cm, below of=a] {${\bf S}/{\mathcal L}$};
\node (b') [node distance=1.6cm, below of=b] {${\bf S}/{\mathcal D}$};
\path[>>->]
(a) edge node[above]{} (b)
(a) edge node[left]{} (a')
(a') edge node[above]{} (b')
(b) edge node[below right]{} (b');
\end{tikzpicture},
\end{equation}
\end{center}
the common composition being the canonical map ${\bf S}\to{\bf S}/{\mathcal D}$, so that
${\bf S}\simeq {\bf S}/{\mathcal R}\times_{{\bf S}/{\mathcal D}}{\bf S}/{\mathcal L}$ This is the {\em Kimura Factorization Theorem}, extended from regular bands to skew lattices. (See \cite{L3} and also \cite{CV1}.) Both factor maps ${\bf S}/{\mathcal R}\to {\bf S}/{\mathcal D}$ and ${\bf S}/{\mathcal L}\to {\bf S}/{\mathcal D}$, moreover, are universal maps from ${\bf S}/{\mathcal R}$ and ${\bf S}/{\mathcal L}$, respectively, to the variety of lattices. Put otherwise, we have the induced isomorphisms $({\bf S}/{\mathcal R})/{\mathcal D}_{{\bf S}/{\mathcal R}} \simeq {\bf S}/{\mathcal D} \simeq ({\bf S}/{\mathcal L})/{\mathcal D}_{{\bf S}/{\mathcal L}}$.

A skew lattice ${\bf S}$ is {\em symmetric} if $a\vee b = b\vee a$ if and only if $a\wedge b = b\wedge a$ for all $a, b \in S$, thus making instances of commutation unambiguous. It is {\em distributive} if both $x\wedge (y\vee z)\wedge x = (x\wedge y
\wedge x) \vee (x\wedge z\wedge x)$ and $x\vee (y\wedge z)\vee x = (x\vee y\vee x)\wedge (x\vee z\vee x)$ hold. It is {\em normal} if $x\wedge y\wedge z\wedge w = x\wedge z\wedge y\wedge w$ holds. In terms of semigroup theory, normality means that $(S,\wedge)$ is a normal band, i.e., a band that is locally a semilattice. Normality is equivalent to
$\lceil e\rceil =\{x\in S\colon e\geq x\}$
being commutative for each $e \in S$ and thus forming a sublattice of~${\bf S}$.  (Normality was studied in~\cite{L6}.)

\begin{example}\label{ex:1}
{\em
Let ${\mathcal{P}}(A, B)$ denote the set of all partial functions from a given set $A$ to a second set $B$. Given functions $f\colon F\to B$ and $g\colon G\to B$ where $F, G\subseteq A$, we define partial functions $f\wedge g$ and $f\vee g$ in ${\mathcal{P}}(A, B)$ as follows:  $f\wedge g=f|_{G\cap F}$ and $f\vee g=g\cup (f|_{F\setminus G})$,  where $f|_{H}$ denotes $f$ restricted to a subset $H$ of $F$. The element $f\vee g$ is often called the {\em override} since $g$ overrides $f$ on their common subdomain. Thus $f\vee g$ favors $g$, and the {\em restriction} $f\wedge g$ favors $f$. The algebra ${\bf{\mathcal{P}}}_L(A,B) = ({\mathcal{P}}(A, B);  \wedge,\vee)$ is a left-handed skew lattice that is distributive, normal and symmetric. In fact it is {\em strongly distributive} in that it satisfies:
$$
x\wedge (y \vee z) = (x\wedge y) \vee (x\wedge z) \, \text{ and } \, (x \vee y)\wedge z = (x\wedge z) \vee (y\wedge z).
$$}
\end{example}

This leads us to recall the following result due to Leech:

\begin{theorem}[{\cite{L6}}]
A skew lattice ${\bf S}$ is strongly distributive if and only if  it is symmetric, distributive and normal.
\end{theorem}

\begin{remark}{\em
The four distributive identities (the two characterizing strong distributivity and the two characterizing distributivity) are mutually equivalent for lattices. This is not so for skew lattices. In particular, the two identities characterizing distributivity are not equivalent in general, but are so for symmetric skew lattices. (See \cite{S1} and \cite{CV3}.)}
\end{remark}

\begin{remark}{\em
Stone duality theory for strongly distributive skew lattices with zero has been developed that extends Priestley duality for distributive lattices with zero, itself an extension of classical Stone duality. (See \cite{BCVGGK}.)}
\end{remark}

Two other operations can be defined on ${\mathcal{P}}(A, B)$, turning ${\bf{\mathcal{P}}}_L(A,B)$ into a variant of the
Boolean algebra on the power set ${\mathcal{P}}(A)$: the {\em difference}, $f \setminus g = f|_{F\setminus G}$, and the nullary
operation given by the empty function $\varnothing$. The latter is the zero of ${\bf{\mathcal{P}}}_L(A,B)$. In general, a {\em zero}
element $0$ of a skew lattice is characterized by the identities:  $0\wedge x = 0 = x\wedge 0$ and $0\vee x = x = x\vee 0$. Zero elements, when they exist, are unique and form a sole ${\mathcal D}$-class. If $x\wedge y = 0$,
then $y\wedge x = 0$, and also $x\vee y = y\vee x$ when ${\bf S}$ is symmetric. In general, if $x\wedge y = 0$ and $u \mathrel{\mathcal{D}} x$ and $v \mathrel{\mathcal{D}} y$, then $u\wedge v = 0$.

A {\em skew Boolean algebra} (sometimes abbreviated SBA) is an algebra ${\bf S} = (S;  \wedge, \vee, \setminus, 0)$ such that the $(\wedge, \vee, 0)$- reduct is a strongly distributive skew lattice with zero element $0$ and $\setminus$ is a binary operation on $S$ satisfying
$$
(x\wedge y\wedge x) \vee (x\setminus y) = x \text{ and } (x\wedge y\wedge x)\wedge (x\setminus y) = 0 = (x\setminus y) \wedge (x\wedge y\wedge x).
$$
By symmetry one has $(x\setminus y) \vee (x\wedge y\wedge x) = x$ also so that $x\wedge y\wedge x$ commutes with $x\setminus y$. This all implies that, for each $e\in {\bf S}$, $\lceil e\rceil = \{x \in S\colon e \geq x\}$ is a Boolean sublattice of ${\bf S}$ with $e\setminus f$ being the unique complement of $e\wedge f\wedge e$ in $\lceil e\rceil$. 

\begin{theorem}[{\cite[Theorem 1.8]{L2}}] Skew Boolean algebras form a variety.
\end{theorem}

$({\mathcal{P}}(A, B); \wedge, \vee,\setminus, \varnothing)$ is an SBA, also denoted ${\bf{\mathcal{P}}}_L(A,B)$. As with Boolean algebras and their power set examples, every left-handed skew Boolean algebra can be embedded in some partial function algebra $({\mathcal{P}}(A, B); \wedge, \vee,\setminus, \varnothing)$. (See Leech \cite[Corollary 1.14]{L2}.)

For every left-handed skew lattice (or skew Boolean algebra), a dual right-handed counterpart exists, and conversely. Indeed, given any skew lattice ${\bf S} = (S; \wedge, \vee)$, its dual algebra ${\bf S}^* = (S; \wedge^*, \vee^*)$ is defined on $S$ by $x\wedge^*y = y\wedge x$ and $x\vee^*y = y\vee x$, with the double-dual ${\bf S}^{**}$ being ${\bf S}$. One algebra is left-handed iff the other is right-handed with both clearly being term equivalent. Thus for any general statement about left-handed skew lattices or SBAs, a dual statement about the right-handed cases holds, and conversely.

 \begin{remark}{\em Given a skew Boolean algebra, the canonical congruences ${\mathcal D}$, ${\mathcal R}$ and ${\mathcal L}$ are skew Boolean algebra congruences, and the skew Boolean algebra versions of the Clifford-McLean Theorem and the Kimura Factorization hold.}
  \end{remark}

\section{Primitive Algebras and Orthogonal Decompositions}\label{s2:primitive}

A skew Boolean algebra is {\em primitive} if it consists of two ${\mathcal D}$-classes, $A > \{0\}$. Given a rectangular skew lattice ${\bf A}$, if $A^0 = A\cup\{0\}$ where $0\not\in A$, then a unique primitive SBA ${\bf{A^0}}$ is defined
on $A^0$ by extending the operations on $A$ by letting $0$ be the zero element and setting
$$
x\setminus y = \left\lbrace\begin{array}{ll}x, & \text{if } y=0;\\
0, &\text{otherwise.}\end{array}\right.
$$
Essentially, all primitive skew Boolean algebras arise in this fashion. Three important primitive algebras
are the following: the generalized Boolean algebra ${\bf 2}$ determined on $\{1, 0\}$ by $1 > 0$; the left-handed primitive skew Boolean algebra ${\bf 3}_L$ determined on $\{1, 2, 0\}$ by $1\mathrel{\mathcal L} 2$ with both $1, 2 > 0$; and ${\bf 3}_R$, its right-handed dual
where $1\mathrel{\mathcal R} 2$. In detail, for ${\bf 3}_L$ we have the Cayley tables:
\vspace{0.2cm}
\begin{center}
\begin{tabular}{ccccc}

\begin{tabular}{c|c:c:cc}
$\wedge$ & $0$ & $1$& $2$ &\\
\cline{1-4}
$0$ & $0$ & $0$ & $0$&\\
\cdashline{1-4}
$1$ & $0$ & $1$ & $1$& \\
\cdashline{1-4}
$2$ & $0$ & $2$ & $2$&
\end{tabular}
&&
\begin{tabular}{c|c:c:cc}
$\vee$ & $0$ & $1$& $2$ &\\
\cline{1-4}
$0$ & $0$ & $1$ & $2$&\\
\cdashline{1-4}
$1$ & $1$ & $1$ & $2$& \\
\cdashline{1-4}
$2$ & $2$ & $1$ & $2$&
\end{tabular}
&&
\begin{tabular}{c|c:c:cc}
$\setminus$ & $0$ & $1$& $2$ &\\
\cline{1-4}
$0$ & $0$ & $0$ & $0$&\\
\cdashline{1-4}
$1$ & $1$ & $0$ & $0$& \\
\cdashline{1-4}
$2$ & $2$ & $0$ & $0$&
\end{tabular}
\end{tabular}
\end{center}

\vspace{0.2cm}
   \begin{theorem}[{\cite[Theorem 1.13]{L2}}]\label{th:2.1}
Every nontrivial skew Boolean algebra is a subdirect product of copies of ${\bf 2}$, ${\bf 3}_L$  and ${\bf 3}_R$; every nontrivial left-handed (right-handed) skew Boolean algebra is a subdirect product of copies of ${\bf 2}$ and ${\bf 3}_L$  (${\bf 2}$ and ${\bf 3}_R$).
\end{theorem}

\begin{remark}
{\em One canonical way of embedding left-handed skew Boolean algebras into powers of ${\bf 3}_L$ is given in~\cite{Kud1}
by the unit of an adjunction between Boolean spaces and left-handed skew Boolean algebras.}
\end{remark}

As a direct consequence of Theorem \ref{th:2.1} we have:
\begin{corollary} An equation or equational implication holds on all skew Boolean algebras iff it holds on ${\bf 3}_L$
and ${\bf 3}_R$. It holds on all left-handed (right-handed) skew Boolean algebras iff it holds on ${\bf 3}_L$ (${\bf 3}_R$) .
\end{corollary}

\begin{corollary} \label{cor:identities}
The following identities hold for skew Boolean algebras:
\begin{enumerate}[(i)]
\begin{multicols}{2}
\item $(x\wedge y)\setminus z = (x\setminus z)\wedge (y\setminus z)$;
\item $(x\vee y)\setminus z = (x\setminus z)\vee (y\setminus z)$;
\item $x\setminus (y\vee z)=(x\setminus y)\wedge (x\setminus z)$;
\item $x\setminus (y\wedge z) = (x\setminus y)\vee (x\setminus z)$;
\item $(x\setminus y)\setminus z = (x\setminus z)\setminus y$;
\columnbreak
\item $(x\setminus y)\setminus z=x\setminus (y\vee z)=x\setminus (z\vee y)$;
\item $x\setminus (x\setminus  y)=x\wedge y\wedge x$;
\item $(x\setminus y)\vee y=y\vee x\vee y=y\vee (x\setminus y)$;
\item $x\setminus (x\wedge y)=x\setminus (y\wedge x)=x\setminus y$.
\end{multicols}
\end{enumerate}
\end{corollary}

Note that (iii) and (iv) are the de Morgan laws for SBAs. While  (i) - (ix) can be checked out on ${\bf 3}_L$ and ${\bf 3}_R$ thanks to Corollary \ref{cor:identities}, that corollary itself depends on a subset of these assertions being derived from the definitions of a skew Boolean algebra. Since (pre-)order relations can be expressed as equalities, we also have the implications:
$$
{\emph(x)}\,\, y\preceq z \Rightarrow x\setminus z\geq x\setminus y; \,\,\,\, {\emph{(xi)}} \,\, x\leq y \Rightarrow x\setminus z\leq y\setminus z; \,\,\,\,
{\emph{(xii)}} \,\,  x\preceq y \Rightarrow x\setminus z\leq y\setminus z.
$$

\begin{corollary}\label{for:lf} Skew Boolean algebras are locally finite in that finite subsets generate finite subalgebras. \end{corollary}

\begin{proof}
Indeed, given a finite set $X$, only finitely many distinct functions exist from $X$ to ${\bf 2}$ or to ${\bf 3}_L$ or to ${\bf 3}_R$. Hence, any embedding from $X$ into a product of these three algebras can be reduced to an embedding into a finite subproduct of these algebras.
\end{proof}
Just as finite (generalized) Boolean algebras are isomorphic to direct products of ${\bf 2}$, more generally finite skew Boolean algebras or even skew Boolean algebras with finitely many ${\mathcal D}$-classes, are isomorphic
to direct products
of finitely many primitive skew Boolean algebras. (See \cite[Theorem 1.16]{L2}  or \cite[Theorem 3.1]{LS}.) In what follows we will be interested in direct products of finitely many primitive skew Boolean algebras. We begin as follows:

Two elements $a$ and $b$ in an SBA are {\em orthogonal} if $a\wedge b = 0$, which is equivalent to  $b\wedge a = 0$ and also to $a\setminus b=a$ and $b\setminus a=b$. These conditions imply $a\vee b=b\vee a$. A set of elements $\{a_1, \dots, a_n\}$ is an {\em orthogonal set} if the $a_i$ are pairwise orthogonal, in which case $a_1\vee a_2\vee\dots\vee a_n = a_{\sigma(1)}\vee a_{\sigma(2)}\vee \dots \vee a_{\sigma(n)}$ for all permutations $\sigma$ on $\{1,2,\dots, n\}$. In this situation, $a_1\vee a_2\vee\dots\vee a_n$ is denoted by $a_1 +\dots +a_n$ or $\sum_1^n a_i$. (In fact, such notation assumes orthogonality.) Such a sum is referred to as an {\em orthogonal sum}, or an {\em orthosum} for short.

A family of ${\mathcal D}$-classes $\{D_1,\dots, D_r\}$ is {\em orthogonal} when elements from distinct classes are orthogonal. For this it is sufficient that some transversal set $\{d_1,\dots, d_r\}$ be
orthogonal. In general we have:

\begin{lemma} Given an orthogonal family of  ${\mathcal D}$-classes $\{D_1,\dots, D_r\}$ and two
orthosums $a_1 +\dots +a_r$ and $b_1 +\dots +b_r$ where $a_i,b_i\in D_i$:
\begin{enumerate}[(i)]
\item $ (a_1 +\dots +a_r) \vee   (b_1 +\dots +b_r) = (a_1\vee b_1) + (a_2\vee b_2) + \dots + (a_r\vee b_r)$;
\item $ (a_1 +\dots +a_r) \wedge   (b_1 +\dots +b_r) = (a_1\wedge b_1) + (a_2\wedge b_2) + \dots + (a_r\wedge b_r)$;\item $ (a_1 +\dots +a_r) \setminus   (b_1 +\dots +b_r) = (a_1\setminus b_1) + (a_2\setminus b_2) + \dots + (a_r\setminus b_r)$;
\item $a_1 +\dots +a_r= b_1 +\dots +b_r$  if and only if $a_i = b_i$ for all $1\leq i\leq r$.
\end{enumerate}
In the case where all $D_1,\dots, D_r$ are nonzero, {(i) -- (iv)} extend to the subalgebra $\sum_1^r D_i^0=\{x_1+\dots +x_r\colon x_i\in D_i^0\}$ generated from
the union $D_1\cup\dots\cup D_r$ where elements from distinct $D_i^0$ are also orthogonal.
\end{lemma}

\begin{proof} The orthosums on both sides of parts (1), (2) and (3) are well-defined because the family $\{D_1,\dots, D_r\}$ is orthogonal. Observe that the conditions $a_i\mathrel{\mathcal D} b_i$ are expressible as equalities. Furthermore,  the condition that $D_i$ is orthogonal to $D_j$, for $i\neq j$, is equivalent to $a_i\wedge a_j=0$, which is also an equality.
It follows that the equalities in parts  (1), (2) and (3), under the given assumptions, are expressible as quasi-identities (where $+$ must be read as $\vee$). Likewise, the claim of (4), under the given assumptions, consists of two quasi-identities. So in order to prove all the claims, it is enough to verify that they hold in ${\bf 3}_R$ and ${\bf 3}_L$. But in this case there is only one non-zero ${\mathcal D}$-class, so that $r=1$, in which case parts (1) - (4) trivialize respectively down to: $a_1\vee b_1=a_1\vee b_1$; $a_1\wedge b_1=a_1\wedge b_1$; $a_1\setminus b_1=a_1\setminus b_1$; and $a_1=b_1$ if and only if $a_1=b_1$. The final statement is immediate, as elements of $\sum_1^r D_i^0$ are orthosums  $x_1+\dots +x_r$ where $x_i\in D_i^0$.
\end{proof}

The algebra $\sum_1^r D_i^0$ is an internal direct product of the primitive subalgebras $D_1,\dots, D_r$. It is also called the {\em orthosum} of the $D_i^0$. Of course, $a\setminus b=0$ whenever $a \mathrel{\mathcal D} b$. A special
case occurs when $D_1,\dots, D_r$ are atomic ${\mathcal D}$-classes, i.e., lying directly over the zero class $\{0\}$. Here meets $a\wedge b$ of elements from distinct classes are forced to be zero, making $a$ and $b$ orthogonal. In general, non-zero orthogonal ${\mathcal D}$-classes $D_1,\dots, D_r$ are the atomic ${\mathcal D}$-classes of the subalgebra $\sum_1^r D_i^0$. In any case, we have the following basic result for SBAs with only finitely many ${\mathcal D}$-classes. (See \cite[Lemma 1.11 and Theorem 1.16]{L2}.)
\begin{theorem}\label{th:2.6}
A nontrivial skew Boolean algebra ${\bf S}$ with finitely many ${\mathcal D}$-classes has a finite number of atomic ${\mathcal D}$-classes $D_1,\dots, D_r$, in which case it is the orthosum $\sum_1^r D_i^0$ of the corresponding primitive subalgebras $D_1^0,\dots, D_r^0$.
\end{theorem}
The above decomposition is an {\em internal} form of the {\em atomic decomposition} of a
skew Boolean algebra ${\bf S}$, which must occur when ${\bf S}/{\mathcal D}$ is finite. This internal form is, of
course, unique. The {\em external} form, given as a direct product, is unique to within isomorphism. In the left-handed case for finite ${\bf S}$, the {\em standard atomic decomposition} is
$$
{\bf S}\simeq {\bf n_1}_{L}\times {\bf n_2}_{L} \times \dots \times {\bf n_r}_{L} \text{ with } 2\leq n_1\leq\dots\leq n_r,
$$
with ${\bf n}_{L}$ being the unique left-handed primitive algebra on $\{0, 1, 2, \dots, n-1\}$ where $0$ is
the $0$-element. Standard decompositions are also unique. Consider ${\bf  2} \times {\bf  2} \times {\bf  4}_L \times {\bf  5}_L \times {\bf  5}_L$ or more briefly ${\bf  2}^2 \times  {\bf  4}_L \times {\bf  5}_L^2$. In this instance ${\bf  2}^2$ provides the center of the algebra
where `$L$' is superfluous. Similar remarks hold in the right-handed case. In the two-sided general case one uses notation such as ${\bf  3}_L\bullet {\bf  5}_R$ to represent the primitive algebra ${\bf  3}_L\times_{\bf 2} {\bf  5}_R$ given by the fibered product, as in:
${\bf S}\simeq {\bf  2}^3 \times ({\bf  3}_L\bullet {\bf  5}_R) \times ({\bf  5}_L\bullet {\bf  4}_R)\times ({\bf  7}_L\bullet {\bf  7}_R)$.
In this case a standard decomposition could be given by lexicographically ordering the factors. In any case, a finite SBA is classified when its standard atomic decomposition is given. Standard atomic decompositions for finite free algebras are determined in Section \ref{s3:finite}.

\begin{example} {\em Partial function algebras serve as primary examples of SBAs. Note that
\begin{multline*}
{\bf{\mathcal{P}}}_L(\{1, \dots, n\}, \{1, \dots, m\})\simeq \prod_{i=1}^n {\bf{\mathcal{P}}}_L(\{i\}, \{1, \dots, m\})\\
\simeq \big({\bf{\mathcal{P}}}_L(\{1\}, \{1, \dots, m\})\big)^n \simeq {\bf{(m+1)}}_L^n.
\end{multline*}
In particular, ${\bf{\mathcal{P}}}_L(\{1, \dots, n\}, \{1\}) \simeq {\bf 2}^n$. In this case each partial function $f$ is determined
by choosing a subset of $\{1, 2, \dots, n\}$ to be $f^{-1}(1)$. This results in a bijection between ${\bf{\mathcal{P}}}_L(\{1, \dots, n\}, \{1\})$ and the power set of $\{1, 2, \dots, n\}$ that preserves the GBA operations.}
\end{example}

In the following lemma we collect some properties of congruences on and homomorphisms of skew Boolean algebras with finitely many ${\mathcal D}$-classes.

\begin{lemma}\label{lem:l:cong}
Let $\theta$ be a congruence on such a skew Boolean algebra $S$ with finitely many ${\mathcal D}$-classes viewed as an
orthosum $\sum_1^r D_i^0$ of primitive subalgebras $D_i^0$, the $D_i$ being the atomic ${\mathcal D}$-classes. Then:
\begin{enumerate}
\item \label{aj1} If $d \mathrel{\theta} 0$ for $d \in D_i$, then $D_i \subseteq [0]_{\theta}$, the congruence class of $0$.
\item  \label{aj2} If $d_1 \mathrel{\theta}  d_2$ with $d_1\in D_i$, $d_2\in D_j$ but $D_i \neq D_j$, then $D_i\cup D_j \subseteq [0]_{\theta}$.
\item  \label{aj3} If some $D_i \subseteq [0]_{\theta}$, then upon re-indexing $D_1, \dots , D_r$, one has $D_1\cup\dots \cup D_k \subseteq [0]_{\theta}$, while the remaining ${\theta}$-classes refine those ${\mathcal D}$-classes that remain. Thus
$$\left (\sum_1^r D_i^0\right)/\theta \simeq \sum_{k+1}^r D_i^0/\theta_i,
$$
where $\theta_i=\theta|_{D_i^0\times D_i^0}$ and $D_i^0/\theta_i$ is primitive for each $i\geq k+1$.
\end{enumerate}

Thus, given ${\bf S} = \sum_1^r D_i^0$ and a homomorphism $f \colon {\bf  S} \to  {\bf  T}$ of skew Boolean algebras:
\begin{enumerate}
\item[(4)]  \label{aj4} $f({\bf S})$ is an orthosum with summands $f (D_i^0)$, each of which is either primitive or
else just $\{0_{{\bf T}}\}$. In the former case, $f(D_i)$ is atomic in  $f({\bf S})$.
\item[(5)]   \label{aj5} In particular, $f (D_i^0)\cap f (D_j^0)\neq \{0_{{\bf T}}\}$ implies $i=j$
and  $D_i^0 = D_j^0$.
\item[(6)]  \label{aj6} Given left-(right-)handed primitive SBAs $D_1^0$ and $D_2^0$, a non-zero homomorphism
from $D_1^0$ to $D_2^0$ is any map sending $0$ to $0$, and elements in $D_1$ to elements in $D_2$.
\item[(7)]  \label{aj7} If $D_1$ and $D_2$ are neither both left-handed nor both right-handed, all non-zero homomorphisms $f \colon D_1^0 \to D_2^0$ are obtained as follows:
\begin{enumerate}[(a)]
\item $f(0)=0$.
\item Pick $a\in D_1$, $b\in D_2$, any maps $\lambda\colon L_a \to L_b$, $\rho\colon R_a \to R_b$.
\item Finally,  for all $x \in L_a$ and $y \in R_a$ set $f(x\wedge y) = \lambda(x) \wedge \rho(y)$.
\end{enumerate}
\end{enumerate}
\end{lemma}

\begin{proof} \eqref{aj1} should be clear. For \eqref{aj2} note that $d_1 = d_1\wedge d_1 \mathrel{\theta} d_1\wedge d_2 =0$ and likewise $d_2\mathrel{\theta} 0$. \eqref{aj3} should now be clear, and (4) and (5) follow from \eqref{aj3}. To see (6), note that such a map describes how  homomorphisms which send $0$ to $0$ of left (right) normal $\wedge$-band reducts (which are of left (right) zero semigroups with bottom element $0$ adjoint)  are obtained in general. That it is a homomorphism of skew latices and hence of skew Boolean algebras follows from $x\wedge y=y\vee x$ holding on ${\mathcal{D}}$-classes and properties of $0$. The statement in (7) describes precisely how homomorphisms that send $0$ to $0$ for the normal $\wedge$-band reduct are obtained in general. Again, such a homomorphism must be a homomorphism of skew Boolean algebras. Note that $f$ is well defined, since each $d\in D_1$ is the unique $x\wedge y$ where $x\in L_a$ and $y\in R_a$ ($x=d\wedge a$ and $y=a\wedge d$).
\end{proof}

Of particular interest in Section \ref{s4:free_min_gen_set} are epimorphisms between finite SBAs, which as functions are surjective. (The argument quickly reduces to the primitive case and then
to showing that epimorphisms of rectangular bands - term equivalents of rectangular skew lattices - are surjective, an easy exercise.) Based on Lemma \ref{lem:l:cong}, and particularly on part (6), epimorphisms between orthosums in general arise as follows:

\begin{proposition}\label{prop:2.8}
Given orthosums of primitive algebras ${\bf S} = \sum_1^r  D_{1,i}^0$ and ${\bf T} = \sum_1^r  D_{2,j}^0$,
an epimorphism $f \colon {\bf S} \to {\bf T}$, should it exist, is obtained as follows:
\begin{enumerate}[(i)]
\item Let $f$ send $0_{{\bf S}}$ and all elements in a (possibly empty) union of a subset of the $D_{1,i}$ to $0_{{\bf T}}$.
\item Next, let $f$ send each remaining $D_{1,i}$ homomorphically, as a skew lattice, onto a distinct $D_{2, j}$, with each $D_{2, j}$ being one such image.
\item Finally, let $f$ extend the resulting partial map to all of ${\bf S}$ by orthosums.
\end{enumerate}
\end{proposition}

In the left-handed case, Proposition \ref{prop:2.8}  follows from the description of morphisms in the duality theory of SBAs \cite{Kud,KudLaw} applied to the case where ${\bf S}/{\mathcal D}$ is finite and thus its dual topology is discrete.

We turn to the question when the epimorphism $f$ of Proposition~\ref{prop:2.8} exists. In terms of standard decompositions we have the following simple covering criteria. We begin with the left-handed case. 

Given left-handed algebras
$$
{\bf S}\simeq {\bf m_1}_{L}\times \dots \times {\bf m_r}_{L} \,\, \text{ and } \,\, {\bf T}\simeq {\bf n_1}_{L}\times \dots \times {\bf n_s}_{L}
$$
with $m_1 \leq m_2 \leq \dots \leq m_r$ and $n_1 \leq n_2 \leq \dots \leq n_s$, it is clear that an
epimorphism $f \colon {\bf S} \to {\bf T}$ exists if and only if:
\begin{enumerate}
\item $r \geq s$.
\item $m_r \geq n_s$, $m_{r-1} \geq n_{s-1}$, $m_{r-2} \geq n_{s-2}$, $\dots$, $m_{r-(s-1)} \geq n_1$.
\end{enumerate}

Thus in terms of atomic ${\mathcal D}$-classes, ${\bf S}$ has at least as many atomic classes as ${\bf T}$, and
each atomic class in ${\bf T}$ can be matched with a distinct atomic class in ${\bf S}$ of equal or greater size. These
conditions are also sufficient to guarantee that ${\bf T}$ can be isomorphically embedded in ${\bf S}$. They are not necessary however since ${\bf 4}_L$ can be embedded in ${\bf 3}_L\times {\bf 3}_L$.

The right-handed case is dual. Consider now the general case where
$$
{\bf S}\simeq ({\bf m_1}_{L}\bullet  {\bf m'_1}_{R})\times \dots \times ({\bf m_r}_{L}\bullet  {\bf m'_r}_{R})\,\, \text{ and } \,\, {\bf T}\simeq ({\bf n_1}_{L}\bullet {\bf n'_1}_{R})\times \dots \times ({\bf n_s}_{L}\bullet {\bf n'_s}_{R}).
$$
Here for an epimorphism $f\colon {\bf S}\to {\bf T}$ to exist it is necessary and sufficient that:
\begin{enumerate}
\item $r \geq s$.
\item Each ${\bf n_j}_{L}\bullet {\bf n'_j}_{R}$ factor of ${\bf T}$ can be matched off with a unique ${\bf m_i}_{L}\bullet  {\bf m'_i}_{R}$ factor of ${\bf S}$ such that $m_i \geq n_j$ and $m'_i\geq n'_j$.
\end{enumerate}

When a surjective homomorphism $f\colon {\bf S}\to {\bf T}$  exists, ${\bf S}$ is said to {\em cover} ${\bf T}$. ${\bf S}$ is a {\em free cover} of ${\bf T}$ if ${\bf S}$ is free and covers ${\bf T}$, and is of minimal size amongst such covers. For finite SBAs, their free covers are unique to within isomorphism. In the Section \ref{s4:free_min_gen_set} we will be interested in free covers of finite SBAs.

\section{Free algebras: the finite case}\label{s3:finite}
Given a non-empty set $X$:
$$ 
\begin{array}{ll}
\,\,\, {\bf{SBA}}_X & \text{is the free skew Boolean algebra on } X.\\
_{\mathcal R}{\bf{SBA}}_X & \text{is the free right-handed skew Boolean algebra on } X.\\
_{\mathcal L}{\bf{SBA}}_X & \text{is the free left-handed skew Boolean algebra on } X.\\
\,\,\,{\bf{GBA}}_X & \text{is the free generalized Boolean algebra on } X.\\
\end{array}
$$

 Free algebras are, of course, unique to within isomorphism. Thus if we say `the free' we have in mind a particular concrete instance, from which we are free (in an alternative sense) to find other isomorphic variants. In this paper, the default free algebra ${\bf{F}}_X$ on an alphabet $X$ is the algebra of all terms (or polynomials) in $X$. In the current context, the terms are defined inductively as follows.
\begin{enumerate}
\item Each $x$ in $X$ is a term, as is the constant $0$.
\item If $u$ and $v$ are terms, so are $(u\vee v)$, $(u\wedge v)$ and $(u\setminus v)$.
\end{enumerate}
Two terms, $u$ and $v$, are equivalent in ${\bf{F}}_X$ if and only if  $u = v$ is an identity in the given variety of algebras. Clearly these criteria for equivalence differ among the four varieties of interest. Given an SBA equation of terms in $X$, $u = v$, one can check if it is an $_{\mathcal L}{\mathrm{SBA}}$ identity (or an $_{\mathcal R}{\mathrm{SBA}}$ identity) by seeing if it holds for all evaluations on ${\bf 3}_L$ (or on ${\bf 3}_R$). It is an SBA identity precisely when it holds for all evaluations on both ${\bf 3}_L$ and ${\bf 3}_R$. Finally, it is a GBA identity if and only if it holds for all evaluations on ${\bf 2}$. In our considerations, we are free to relax aspects of the syntax for parentheses if all ways of reinserting them lead to equivalent expressions. E.g., that would happen with $x\vee y\vee z$, but not with $x\wedge y\vee z$.

Given the universal character of the homomorphisms involved in the Clifford-McLean and the Kimura Factorization theorems for skew Boolean algebras, we have:
$$
\begin{array}{c} \vspace{0.1cm}
{\bf{GBA}}_X \, \simeq  \, {\bf{SBA}}_X/{\mathcal D} \, \simeq \, _{\mathcal R}{\bf{SBA}}_X/{\mathcal D}\,  \simeq \, _{\mathcal L}{\bf{SBA}}_X/{\mathcal D},\\ 

\vspace{0.1cm}
_{\mathcal R}{\bf{SBA}}_X \, \simeq \, {\bf{SBA}}_X/{\mathcal L}  \,\, \text{ and }\,\, _{\mathcal L}{\bf{SBA}}_X \, \simeq \, {\bf{SBA}}_X/{\mathcal R},\\ 

{\bf{SBA}}_X\simeq \, _{\mathcal L}{\bf{SBA}}_X\times_{{\bf{GBA}}_X } \, _{\mathcal R}{\bf{SBA}}_X.
\end{array}
$$
(Let ${\mathcal V}$ be any variety of algebras with ${\mathcal W}$ a subvariety of ${\mathcal V}$. For each algebra ${\bf A}$ of ${\mathcal V}$, let $\theta_{\bf A}$ be the congruence on ${\bf A}$ such that ${\bf A}/\theta_{\bf A}$ is in ${\mathcal W}$ and the induced map $\varphi_{\bf A}\colon {\bf A} \to {\bf A}/\theta_{\bf A}$ is a universal homomorphism from ${\bf A}$ to ${\mathcal W}$. Then if ${\bf A}$ is a free ${\mathcal V}$-algebra on generating set $X$, then ${\bf A}/\theta_{\bf A}$ is a free ${\mathcal W}$-algebra on generating set $\varphi_{\bf A}(X)$. In the above context, $\theta_{\bf A}={\mathcal D}, {\mathcal L}$ or ${\mathcal R}$ as appropriate, with $X$ and $\varphi_{\bf A}(X)$ equipotent under $\varphi_{\bf A}$.)

In what follows we first consider ${\bf{SBA}}_n$, $_{\mathcal L}{\bf{SBA}}_n$,  etc. which denote  ${\bf{SBA}}_X$, $_{\mathcal L}{\bf{SBA}}_X$,  etc.  on alphabet $X = \{x_1, x_2, \dots , x_n\}$. Their standard atomic decomposition is given in Theorem \ref{th:3.3} below. But to obtain the latter we need to understand their atomic structure. The case for $_{\mathcal L}{\bf{SBA}}_n$ and for $_{\mathcal R}{\bf{SBA}}_n$ is described in Theorem \ref{th:3.2}, the content of which is our immediate goal. We focus on $_{\mathcal L}{\bf{SBA}}_n$. Since $_{\mathcal L}{\bf{SBA}}_n$ has finitely
many generators, it is finite and thus is determined by its atomic ${\mathcal D}$-classes. We first describe these classes. Each class consists of atoms all sharing a common form. The justification that they are indeed the atomic ${\mathcal D}$-classes will
follow. They are the $2^n-1$ classes of one of the forms below where $y_1, y_2, \dots , y_n$ in the
table represents an arbitrary permutation of $x_1, x_2, \dots , x_n$. A typical class arises from a partition $\{L|M\}$ of $\{x_1, x_2, \dots, x_n\}$ with $k \geq 1$ elements in $L$ and $n-k$ elements in $M$ used to form the term $$(y_1\wedge  \dots \wedge y_k) \setminus (y_{k+1}\vee \dots \vee y_n).$$ This partition is ordered in that $\{L|M\}$ is distinct from $\{M|L\}$. Thus, e.g., $\{1, 2|3, 4\} \neq \{3, 4|1, 2\}$.
\vspace{0.2cm}
\begin{center}
{\renewcommand{\arraystretch}{1.2}
\begin{tabular}{c|c|c}\renewcommand{\arraystretch}{2}\addtolength{\tabcolsep}{1pt}
Form type & Number of classes of this form & Class size \\[3pt]
\hline
$y_1\setminus (y_2\vee y_3\vee \dots \vee y_n)$ & $n=\binom{n}{1}$ & $1$\\[3pt]
\hdashline
$(y_1\wedge y_2)\setminus (y_3\vee y_4\vee \dots \vee y_n)$ & $\binom{n}{2}$ & $2$\\[3pt]
\hdashline
$(y_1\wedge y_2\wedge y_3)\setminus (y_4\vee y_5\vee \dots \vee y_n)$ & $\binom{n}{3}$ & $3$\\[3pt]
\hdashline
$\dots$ & $\dots$ & $\dots$\\[3pt]
\hdashline
$y_1\wedge y_2\wedge y_3\wedge \dots\wedge y_n$ & $1=\binom{n}{n}$ & $n$
\end{tabular}}
\end{center}
\vspace{0.2cm}

Given the left-handed  identity $x\wedge y\wedge z = x\wedge z\wedge y$ and the two-sided identities
\begin{equation}\label{eq:1}
x\setminus (y\vee z) = x\setminus (z\vee y) = (x\setminus y)\setminus z = (x\setminus z)\setminus y
\end{equation}
(easily checked on ${\bf 3}_L$ or on ${\bf 3}_L$ and ${\bf 3}_R$, respectively), $(y_1\wedge \dots \wedge y_k)\setminus (y_{k+1}\vee \dots \vee y_n)$ is invariant in outcome under any permutation of $y_2, \dots, y_k$ or of $y_{k+1}, \dots, y_n$. What does distinguish the elements in each class is the left-most element or variable, $y_1$. In all, a total of $n2^{n-1}$ essentially distinct atoms exist to produce $n2^{n-1}$ $n$-variable functions on ${\bf 3}_L$ (or on ${\bf 3}_R$). This is verified in the proof of Theorem~\ref{th:3.2} below. But first we provide an example.

\begin{example}\label{ex:3.1} For $X = \{x, y, z, w\}$, the $15$ atomic classes and the $4\cdot 2^3 = 32$ atoms are:
\begin{gather*}
\{x\setminus (y\vee z\vee w)\}, \,\,\, \{y\setminus (x\vee z\vee w)\}, \,\,\, \{z\setminus (x\vee y\vee w)\}, \,\,\, \{w\setminus (x\vee y\vee z)\},\\
\{(x\wedge y)\setminus (z\vee w), (y\wedge x)\setminus (z\vee w)\}, \quad \{(x\wedge z)\setminus (y\vee w), (z\wedge x)\setminus (y\vee w)\},\\
\{(x\wedge w)\setminus (y\vee z), (w\wedge x)\setminus (y\vee z)\}, \quad \{(y\wedge z)\setminus (x\vee w), (z\wedge y)\setminus (x\vee w)\},\\
\{(y\wedge w)\setminus (x\vee z), (w\wedge y)\setminus (x\vee z)\}, \quad \{(z\wedge w)\setminus (x\vee y), (w\wedge z)\setminus (x\vee y)\},\\
\{(y\wedge z\wedge w)\setminus x, (z\wedge y\wedge w)\setminus x, (w\wedge y\wedge z)\setminus x\},\\
\{(x\wedge z\wedge w)\setminus y, (z\wedge x\wedge w)\setminus y, (w\wedge x\wedge z)\setminus y\}, \\
\{(x\wedge y\wedge w)\setminus z, (y\wedge x\wedge w)\setminus z, (w\wedge x\wedge y)\setminus z\}, \\
\{(x\wedge y\wedge z)\setminus w, (y\wedge x\wedge z)\setminus w, (z\wedge x\wedge y)\setminus w\},\\
\{x\wedge y\wedge z\wedge w,  y\wedge x\wedge z\wedge w, z\wedge x\wedge y\wedge w, w\wedge x\wedge y\wedge z\}.
\end{gather*}
\end{example}

\begin{theorem}\label{th:3.2} Given the free left-handed skew Boolean algebra $_{\mathcal L}{\bf{SBA}}_n$ on $\{x_1, \dots , x_n\}:$
\begin{enumerate}[(i)]
\item $_{\mathcal L}{\bf{SBA}}_n$ is a finite algebra whose atoms are the terms $(y_1\wedge \dots \wedge y_k)\setminus (y_{k+1}\vee\dots\vee y_n)$
where $k \geq1$ and $(y_1, \dots , y_n)$ is a permutation of $\{x_1, \dots , x_n\}$.
\item Atoms $(y_1\wedge \dots \wedge y_k)\setminus (y_{k+1}\vee\dots\vee y_n)$ and $(z_1\wedge \dots \wedge z_l)\setminus (z_{l+1}\vee\dots\vee z_n)$ lie in the same atomic class if and only if $k=l$, $(z_1, \dots , z_k)$ is a permutation of $\{y_1, \dots , y_k\}$ and
thus $(z_{l+1}, \dots , z_n)$ is a permutation of $\{y_{k+1},\dots , y_n\}$.
\end{enumerate}
\begin{enumerate}
\item[$(iii)_{\mathcal L}$] $(y_1\wedge \dots \wedge y_k)\setminus (y_{k+1}\vee\dots\vee y_n) = (z_1\wedge \dots \wedge z_l)\setminus (z_{l+1}\vee\dots\vee z_n)$ if besides (ii), $y_1 = z_1$.
\end{enumerate}
For the free right-handed dual algebra $_{\mathcal R}{\bf{SBA}}_n$, (i) and (ii) again hold along with:
\begin{enumerate}
\item[$(iii)_{\mathcal R}$] $(y_1\wedge \dots \wedge y_k)\setminus (y_{k+1}\vee\dots\vee y_n) = (z_1\wedge \dots \wedge z_l)\setminus (z_{l+1}\vee\dots\vee z_n)$ if in addition to~(ii), $y_k = z_k$.
\end{enumerate}
\end{theorem}
\begin{proof} 
We consider the left-handed case. The right-handed assertion is similar.
To begin, given a permutation $(z_1, \dots, z_k)$ of the set $\{y_1, \dots , y_k\}$, $(z_1\wedge \dots \wedge z_k)\setminus (y_{k+1}\vee\dots\vee y_n)$ and
$(y_1\wedge \dots \wedge y_k)\setminus (y_{k+1}\vee\dots\vee y_n)$ are ${\mathcal L}$-related; they are not equal if $z_1 \neq y_1$. Indeed, $y_1\wedge \dots \wedge y_k \mathrel{{\mathcal L}} z_1\wedge \dots \wedge z_k$ plus $(a\setminus c)\wedge (b\setminus c) = (a\wedge b)\setminus c$ implies they are ${\mathcal L}$-related; they are not equal if $z_1 \neq y_1$ since they are not equal when operating as functions on ${\bf 3}_L$. Just give $y_1$ and $z_1$ values $1$ and $2$, respectively, the remaining front variables $1$, and all $n-k$ back variables $0$. The outcome for $(y_1\wedge \dots \wedge y_k)\setminus (y_{k+1}\vee\dots\vee y_n)$ is $1$ and for $(z_1\wedge \dots \wedge z_k)\setminus (y_{k+1}\vee\dots\vee y_n)$ is~$2$.

In general, given distinct partitions $\{L|M\}$ and $\{L'|M'\}$ of $\{x_1, x_2, \dots, x_n\}$ with $L$ and $L'$ non-empty, some element $m$ lies in $L\cap M'$ or in $L'\cap M$, say the former. Viewing $m$ as a generator, given any $\{L|M\}$-term $u$ and any $\{L'|M'\}$-term $v$, we have $u\wedge m = u$ but $v\wedge m = 0 = m\wedge v$. Thus $u\wedge v = u\wedge m\wedge v = u\wedge 0 = 0 = v\wedge u$. Thus all $\{L|M\}$-terms are orthogonal to all $\{L'|M'\}$-terms. Since $(x_1\wedge \dots \wedge x_k)\setminus (x_{k+1}\vee\dots\vee x_n)=0$ is not an identity in ${\bf 3}_L$ for $k \geq 1$, all $\{L|M\}$-classes are non-zero classes and distinct $\{L|M\}$-classes are orthogonal. (Returning to the example above, $\{(x\wedge y)\setminus (z\vee w), (y\wedge x)\setminus (z\vee w)\}$ is disjoint from $\{(x\wedge w)\setminus (y\vee z), (w\wedge x)\setminus (y\vee z)\}$ with pairs from distinct classes being orthogonal.)

To see that they are full ${\mathcal D}$-classes of  $_{\mathcal L}{\bf{SBA}}_n$  and that they are (all the) atomic ${\mathcal D}$-classes, observe first that  they are the atomic ${\mathcal D}$-classes in the subalgebra of  $_{\mathcal L}{\bf{SBA}}_n$  that they generate. We need to show that this subalgebra is in fact all of  $_{\mathcal L}{\bf{SBA}}_n$ . We do so by
showing that each generator $x_k$ of  $_{\mathcal L}{\bf{SBA}}_n$  is in the generated subalgebra. The identities \eqref{eq:1} give us:
\begin{align*}
x_1 & =  (x_1\wedge x_2) + (x_1\setminus x_2) \\
& = (x_1\wedge x_2\wedge x_3) + ((x_1\wedge x_2)\setminus x_3) + ((x_1\setminus x_2)\wedge x_3) + ((x_1\setminus x_2)\setminus x_3)\\
& = (x_1\wedge x_2\wedge x_3) + ((x_1\wedge x_2)\setminus x_3) + ((x_1\wedge x_3)\setminus x_2) + (x_1\setminus (x_2\vee x_3))\\
& = (x_1\wedge x_2\wedge x_3\wedge x_4) + ((x_1\wedge x_2\wedge x_3)\setminus x_4) + \dots
\end{align*}
The process keeps repeating on each new term until generator $x_1$ is resolved into an orthosum of $2^{n-1}$ $\{L|M\}$-type terms - indeed into all the $\{L|M\}$-type terms with leftmost entry $x_1$. Similar calculations work for the remaining generators. Thus the $2^{n-1}$ distinct
$\{L|M\}$-classes are all the atomic ${\mathcal D}$-classes of $_{\mathcal L}{\bf{SBA}}_n$. 

 The right-handed assertion follows from the term equivalence of the two types of algebras.
\end{proof}

In the generalized Boolean case, all atomic terms resulting from the same $\{L|M\}$-decomposition are equated. Thus the particular left-most generator/variable no longer differentiates among outcomes. In the two-sided case, in the Kimura fibered product construction each left-handed atomic class is matched off with the right-handed atomic class with the same $\{L|M\}$ partition. In this case the data of $$((y_1\wedge \dots \wedge y_k), (y'_1\wedge \dots \wedge y'_k))$$ can be combined as $y_1\wedge \dots \wedge y_k\wedge y'_1\wedge \dots \wedge y'_k$ and then reduced via two-sided normality. 

Returning to Example \ref{ex:3.1}, the terms there describe the atomic classes of both the left- and right-handed free algebras on $\{x, y, z, w\}$. Thus $\{(x\wedge y)\setminus (z\vee w), (y\wedge x)\setminus (z\vee w)\}$ works in the left-handed case, while $\{(x\wedge y)\setminus (z\vee w), (y\wedge x)\setminus (z\vee w)\}$ works in the right-handed case. In both finite cases it is possible to describe the `atomic' terms using cyclic permutations in a way that the terms do double duty. But that will not `stretch' to the two-sided case. Here we adjoin both $(x\wedge y\wedge x)\setminus (z\vee w)$ and $(y\wedge x\wedge y)\setminus (z\vee w)$ to the class to get:
$$\{(x\wedge y)\setminus (z\vee w), (y\wedge x)\setminus (z\vee w), (x\wedge y\wedge x)\setminus (z\vee w), (y\wedge x\wedge y) \setminus (z\vee w)\}.$$
For two terms to be equal in value, both end variables in the left part would have to agree. In general, the corresponding atomic classes would be squared in size.

We thus obtain precise structural descriptions of all four relevant free algebras. In what follows $D_{\{L|M\}}$ is the $\{L|M\}$-induced ${\mathcal D}$-class ($={\mathcal L}$-class), ${\bf P}^{\mathcal L}_{\{L|M\}}$ is the left-handed
primitive algebra $D_{\{L|M\}}^0$ and  ${\bf P}^{\mathcal R}_{\{L|M\}}$ is its right-handed counterpart. Also, given
primitive algebras ${\bf P}$ and ${\bf Q}$, ${\bf P}\bullet {\bf Q}$ denotes their {\em fibered product} over ${\bf 2}$, ${\bf P}\times_{{\bf 2}}{\bf Q}$. In the next theorem, the trivial algebra ${\bf 1}$ on $\{0\}$ is included to allow the full distribution of binomial coefficients. This factor corresponds to the front-empty partition $\{\varnothing|X\}$. Hence:

\begin{theorem}\label{th:3.3} The free left-handed skew Boolean algebra $_{\mathcal L}{\bf{SBA}}_n$ on $\{x_1, \dots , x_n\}$ is a direct sum of the primitive algebras ${\bf P}^{\mathcal L}_{\{L|M\}}$ where $\{L|M\}$ ranges over all partitions $\{L|M\}$ of $\{x_1, \dots , x_n\}$ where $L \neq \varnothing$. Thus:
$$
_{\mathcal L}{\bf{SBA}}_n \simeq {\bf 1}^{\binom{n}{0}}\times {\bf 2}^{\binom{n}{1}} \times {\bf 3}_L^{\binom{n}{2}}\times {\bf 4}_L^{\binom{n}{3}} \times\dots \times {\bf{(n+1)}}_L^{\binom{n}{n}}.
$$

Dually, the free right-handed skew Boolean algebra $_{\mathcal R}{\bf{SBA}}_n$ on $\{x_1, \dots , x_n\}$ is a direct sum of the primitive algebras ${\bf P}^{\mathcal R}_{\{L|M\}}$ where $\{L|M\}$ shares the same range. Thus:
$$
_{\mathcal R}{\bf{SBA}}_n \simeq {\bf 1}^{\binom{n}{0}}\times {\bf 2}^{\binom{n}{1}} \times {\bf 3}_R^{\binom{n}{2}}\times {\bf 4}_R^{\binom{n}{3}} \times\dots \times {\bf{(n+1)}}_R^{\binom{n}{n}}.
$$
Finally, the free skew Boolean algebra ${\bf{SBA}}_n$ on $\{x_1, \dots , x_n\}$ is a direct sum of the primitive algebras ${\bf P}^{\mathcal L}_{\{L|M\}}\bullet {\bf P}^{\mathcal R}_{\{L|M\}}$ where $\{L|M\}$ again shares the same range. Thus:
$$
{\bf{SBA}}_n \simeq {\bf 1}^{\binom{n}{0}}\times {\bf 2}^{\binom{n}{1}} \times ({\bf 3}_L\bullet {\bf 3}_R)^{\binom{n}{2}}\times ({\bf 4}_L\bullet {\bf 4}_R)^{\binom{n}{3}} \times\dots \times ({\bf{(n+1)}}_L\bullet {\bf{(n+1)}}_R)^{\binom{n}{n}}.
$$
\end{theorem}
\begin{corollary}\label{cor:3.4}
For all $n \geq 1$:
\begin{enumerate}[(i)]
\item $$
|_{\mathcal L}{\bf{SBA}}_n| = 2^{\binom{n}{1}}3^{\binom{n}{2}}4^{\binom{n}{3}}\dots (n+1)^{\binom{n}{n}}.
$$
\item $$
|{\bf{SBA}}_n|  = 2^{\binom{n}{1}}5^{\binom{n}{2}}10^{\binom{n}{3}}\dots (n^2+1)^{\binom{n}{n}}.
$$
\end{enumerate}
Moreover, if $\alpha_L(n)$, $\alpha_R(n)$ and $\alpha(n)$ denote the number of atoms in $_{\mathcal L}{\bf{SBA}}_n$, $_{\mathcal R}{\bf{SBA}}_n$  and ${\bf{SBA}}_n$, respectively, then:
\begin{enumerate}
\item[(iii)]  $$\alpha_L(n)=\alpha_R(n)=\binom{n}{1}1+\binom{n}{2}2+ \dots +\binom{n}{n-1}(n-1)+\binom{n}{n}n.$$
\item[(iv)] $$\alpha(n)=\binom{n}{1}1+\binom{n}{2}4+\dots +\binom{n}{n-1}(n-1)^2+\binom{n}{n}n^2.$$
\end{enumerate}

\end{corollary}
\begin{proof}  Both (i) and (ii) are clear. Replacing $m$ by $m-1$ and products and powers by (repeated) sums in  (i) calculates the equal number of atoms in $_{\mathcal L}{\bf{SBA}}_n$ and $_{\mathcal R}{\bf{SBA}}_n,$ so that  (iii) follows. In similar fashion, (iv) counts the number of atoms in ${\bf{SBA}}_n$.
\end{proof}
Standard combinatorial arguments give the following simplifications:
\begin{corollary}\label{cor:3.5} Given $\alpha_L(n)$, $\alpha_R(n)$ and $\alpha(n)$ as above:
$$
\alpha_L(n)=\alpha_R(n)=n2^{n-1}\,  \text{ and } \, \alpha(n)=n(n+1)2^{n-2}, \text{ so that } \alpha(n)=\frac{n+1}{2}\alpha_L(n).
$$
\end{corollary}
\begin{proof}
Proof. To see $$\binom{n}{1}1+\binom{n}{2}2+ \dots +\binom{n}{n}n=n2^{n-1},$$ differentiate the binomial expansion $(1+x)^n$ and set $x=1$.
Setting $x =1$ again in the second derivative of the binomial expansion of $(1 + x)^n$ gives
$$
\binom{n}{2}2\cdot 1+ \binom{n}{3}3\cdot 2+ \dots +\binom{n}{n}n(n-1)=n(n-1)2^{n-2}.
$$
Adding the equality of the previous expansion to this and simplifying gives
$$\binom{n}{1}1+\binom{n}{2}4+\dots +\binom{n}{n}n^2=n2^{n-1}+n(n-1)2^{n-2}=n(n+1)2^{n-2}.
$$
\end{proof}
A short table of values follows with the sizes for $n \leq 5$ given to $4$-digit accuracy.

\vspace{0.2cm}
\begin{center}
{\renewcommand{\arraystretch}{1.2}
\begin{tabular}{c|c|c|c|c}\renewcommand{\arraystretch}{2}\addtolength{\tabcolsep}{1pt}
$n$ & $|_{\mathcal L}{\bf{SBA}}_n| $ & $\alpha_L(n) $ & $ |{\bf{SBA}}_n|$ & $\alpha(n)$\\[3pt]
\hline
$2$ & $12$ & $4$ & $20$ & $6$\\[3pt]
\hdashline
$\,\,3\,\,$ & $\,\, 864\,\, $ & $\,\,12\,\,$ & $\,\,10,000\,\,$ & $\,\,24\,\,$\\[3pt]
\hdashline
$4$ & $14,929,920$ & $32$ & $425\cdot 10^8$ & $80$\\[3pt]
\hdashline
$5$ & $3.715\cdot 10^{16}$ & $80$ & $3.017\cdot 10^{25}$ & $240$
\end{tabular}}
\end{center}
\vspace{0.2cm}
\begin{remark} {\em The formulas in Corollary \ref{cor:3.4} (i) and (ii) were found independently by Bignall and Spinks using a general algorithm in computer algebra that gives the sizes of free algebras. (Unpublished communication.)}
\end{remark}

Since
$$\binom{n}{1}+\binom{n}{2}+\dots +\binom{n}{n}=2^n-1,
$$
we have:

\begin{corollary} \label{cor:3.6}
A free (left-handed, right-handed or two-sided) skew Boolean algebra on $n$ generators has $2^n -1$ primitive factors in its atomic decomposition. In general, any
skew Boolean algebra on $n$ generators has $\leq 2^n -1$ primitive factors. Any generalized Boolean algebra on $n$ generators thus has $\leq 2^n -1$ atoms, with the algebra being free if
and only if it has exactly $2^n -1$ atoms, making it isomorphic to ${\bf 2}^{2^n-1}$.
\end{corollary}

Elements in the free algebra are orthosums of atoms in a unique way (modulo commutativity). This includes $0$ as the empty sum. The essentially unique expression of an element as an orthosum of atoms is called its {\em normal form}. For $x\vee y$ this is $(x\setminus y) + (y\wedge x) + (y\setminus x)$ in $_{\mathcal L}{\bf{SBA}}_{\{x,y\}}$. (Actually, $x\vee y = (x\setminus y) + (y\wedge x) + (y\setminus x)$ holds for all SBAs. If $x\mathrel{\mathcal{D}}y$ it
becomes $x\vee y = y\wedge x$.) But when we extend the generating set this changes. Thus in
$_{\mathcal L}{\bf{SBA}}_{\{x,y,z\}}$ the normal form of $x\vee y$ is $$((x\wedge z)\setminus y) + (x\setminus (y\vee z)) + (y\wedge x\wedge z) + ((y\wedge x)\setminus z) + ((y\wedge z)\setminus  x )+
(y \setminus (x\vee z)).$$ If the generating set is extended further to four elements, $\{x, y, z, w\}$, the normal form for $x\vee y$ would grow to $12$ atomic components.

Alternative ways of expressing atoms of free algebras exist and hence also variant representations of these algebras. In the left-handed case, one could express a non-zero atom as $(x, \{L|M\})$ or even $(x, L)$, where $L$ is any non-empty subset of $X$ and $x \in L$. Here:
$$
(x, \{L|M\}) \wedge  (x', \{L|M\}) = (x, \{L|M\}) = (x', \{L|M\}) \vee (x, \{L|M\}).
$$
In the right-handed case, non-zero atoms would look like $(\{L|M\}, x)$, or even $(L, x)$, where
again $x\in L \subseteq X$. Here:
$$(\{L|M\}, x) \wedge (\{L|M\}, x') = (\{L|M\}, x') = (\{L|M\}, x') \vee (\{L|M\}, x).
$$
In the two-sided case, we have $(x, \{L|M\}, y)$, or even $(x, L, y)$, with $x, y \in L\subseteq  X$. Here:
$$(x, L, y) \wedge (x', L, y') = (x, L, y') = (x', L, y') \vee (x, L, y).$$
Thus in these cases the atoms would be parameterized as pointed or doubly pointed non-empty subsets of $X$.

Returning to the main discussion, Theorems \ref{th:2.6} and \ref{th:3.3} lead~to:
\begin{corollary} Every finite skew Boolean algebra is isomorphic to a subalgebra of a finite free skew Boolean algebra. Every finite left-handed (right-handed) skew Boolean algebra is isomorphic to a direct factor of a finite free left-handed (right-handed) skew Boolean algebra.
\end{corollary}
\begin{proof} Indeed, if ${\bf S} \simeq  {\bf 2}^{\mu_2} \times {\bf 3}_L^{\mu_3} \times \dots \times {\bf m}_L^{\mu_m}$, then ${\bf S}$ is isomorphic to a direct
factor of $_{\mathcal L}{\bf{SBA}}_n$ for any $n$ such that
$$
\binom{n}{1}\geq \mu_2,\,  \binom{n}{2}\geq \mu_3,\,\dots, \, \binom{n}{m-1}\geq \mu_m.
$$
\end{proof}

\section{Free algebras and minimal generating sets}\label{s4:free_min_gen_set}

To see what free skew Boolean algebras can tell us about skew Boolean algebras in general, we begin with:

\begin{proposition}\label{prop:4.1}
Given a skew Boolean algebra ${\bf S}$ with generating set $\{a_1,\dots, a_n\}$, let $f\colon {\bf{SBA}}_n \to  {\bf S}$ be the homomorphism onto ${\bf S}$ induced by the map $x_i \mapsto a_i$ for $i \leq n$. Then:
\begin{enumerate}[(i)]
\item The images in ${\bf S}$  under $f$ of the $2^{n} -1$ atomic ${\mathcal D}$-classes of ${\bf{SBA}}_n$ form a family of orthogonal ${\mathcal D}$-classes in ${\bf S}$, with some possibly degenerating to $\{0\}$.
\item The remaining ${\mathcal D}$-classes in this family are the atomic classes of ${\bf S}$, each of which is the image of a unique atomic ${\mathcal D}$-class in ${\bf{SBA}}_n$.
\item Finally, if the images of distinct atoms of ${\bf{SBA}}_n$ are distinct non-zero elements in ${\bf S}$, then ${\bf S}$ is free on $\{a_1,\dots, a_n\}$ and the images under $f$ of the atoms of ${\bf{SBA}}_n$ are
the atoms of~${\bf S}$.
\item In particular, $|S|\leq |{\bf{SBA}}_n|$. Equality occurs if and only if ${\bf S}$ is free on $\{a_1,\dots, a_n\}$.
\end{enumerate}
(Similar statements hold in the left and right-handed cases.)
\end{proposition}

\begin{proof} Clearly, the induced homomorphism goes onto ${\bf S}$. The assertions follow from the
atomic decomposition in ${\bf{SBA}}_n$ and Proposition \ref{prop:2.8}.
\end{proof}

To continue, we introduce the following concept. A generating set of finite (possibly left- or right-handed) skew Boolean algebra ${\bf S}$ is a {\em minimal generating set} of ${\bf S}$ if no other set of generators of ${\bf S}$ has fewer members. Proposition \ref{prop:4.1} gives us:

\begin{theorem}\label{th:4.2} Given the free algebra ${\bf{SBA}}_X$ with $|X|$ finite,  $X$ is a minimal generating set. Conversely, ${\bf{SBA}}_X$ is free on each of its minimal generating sets.  (Left-handed and right-handed variants of both assertions also hold.)
\end{theorem}

\begin{proof} Let $Y$ be another generating set of smaller size than $X$. Then $ |{\bf{SBA}}_X| \leq|{\bf{SBA}}_Y| $ since ${\bf{SBA}}_X$ is now a homomorphic image of ${\bf{SBA}}_Y$. But Theorem \ref{th:3.3} implies $ |{\bf{SBA}}_X| \leq |{\bf{SBA}}_Y| $ cannot be if $|Y| < |X|$. The converse follows from Proposition \ref{prop:4.1} (iii) and (iv). 
\end{proof}

Given a skew Boolean algebra ${\bf S}$, its {\em rank}, denoted by $\rho({\bf S})$, is the least integer $n$ such that ${\bf S}$ can be generated by $n$ generators. Equivalently, $\rho({\bf S})$ is the least $n$ such that ${\bf S}$ is a homomorphic image of ${\bf{SBA}}_n$. In the left- or right-handed cases, $\rho({\bf S})$ is also the least $n$ such that 
${\bf S}$ is a homomorphic image of $_{\mathcal L}{\bf{SBA}}_n$ or $_{\mathcal R}{\bf{SBA}}_n$, respectively. Clearly $\rho({\bf{1}})=0$, $\rho({\bf{2}})=1$ and in general $\rho({\bf{n}}_L)=\rho({\bf{n}}_R)=n-1$ for all $n\geq 1$.
Focusing on the left-handed case, let ${\bf S}$ factor isomorphically as ${\bf 2}^{\mu_2} \times {\bf 3}_L^{\mu_3} \times \dots \times {\bf m}_L^{\mu_m}$ with $\mu_m\neq 0$.
Determining $\rho({\bf S})$ is equivalent to determining the rank of its free cover, $_{\mathcal L}{\bf{SBA}}_n$.
By the covering criteria at the end of Section \ref{s2:primitive}, the following are both necessary and sufficient conditions for $_{\mathcal L}{\bf{SBA}}_n$ to at least have ${\bf S}$ as a homomorphic image:

\begin{enumerate}
\item $_{\mathcal L}{\bf{SBA}}_n$ must have at least $\mu_m$ factors of order $\geq m$.
\item Besides $\mu_m$ of these, $_{\mathcal L}{\bf{SBA}}_n$ must have $\mu_{m-1}$ more factors of order $\geq m-1$.
\item Beyond the $\mu_m + \mu_{m-1}$ of the above factors, $_{\mathcal L}{\bf{SBA}}_n$  must have at least $\mu_{m-2}$ more factors of order $\geq m-2$, etc.
\end{enumerate}

\begin{example} {\em Consider ${\bf S}={\bf 2}^2\times {\bf 3}_L^4\times {\bf 4}_L^3 \times {\bf 5}_L^{48} \times {\bf 6}_L^{11} \times {\bf 7}_L^8$ and assume that $_{\mathcal L}{\bf{SBA}}_n$ can cover ${\bf S}$. Because of the ${\bf 7}_L$-factor, $n$ must be at least $6$. Consider $_{\mathcal L}{\bf{SBA}}_6$ and $_{\mathcal L}{\bf{SBA}}_7$:
$$
_{\mathcal L}{\bf{SBA}}_6\simeq {\bf 2}^6\times {\bf 3}_L^{15}\times {\bf 4}_L^{20}\times {\bf 5}_L^{15}\times {\bf 6}_L^6\times {\bf 7}_L^1 \text{ does not have enough } {\bf 6}_L\text{-} \text{ or } {\bf 7}_L\text{-factors.}
$$
$$
_{\mathcal L}{\bf{SBA}}_7\simeq {\bf 2}^7\times {\bf 3}_L^{21}\times {\bf 4}_L^{35}\times {\bf 5}_L^{35}\times {\bf 6}_L^{21}\times {\bf 7}_L^7\times {\bf 8}_L^1 \text{ has enough } {\bf 6}_L\text{-}, {\bf 7}_L\text{-} \text{ and } {\bf 8}_L\text{-factors.}
$$

In particular, ${\bf 7}_L^8$ is a homomorphic image of ${\bf 7}_L^7\times {\bf 8}_L^1$ and $10$ ${\bf 6}_L$-factors are left over. These $10$ factors and the $35$ ${\bf 5}_L$-factors, however, are not big enough to account for all $48$ needed ${\bf 5}_L$-factors. Only in passing to $_{\mathcal L}{\bf{SBA}}_8$ we do acquire enough factors to account for each factor in ${\bf S}$, whether by isomorphism or by epimorphism. Thus $\rho({\bf S})=8$.}
\end{example}

In general, set 
$$
\Gamma_m^n=\binom{n}{m}+\binom{n}{m+1}+\dots +\binom{n}{n}
$$
for $n\geq m$. $\Gamma_m^n$ counts the number of primitive factors in
$$
_{\mathcal L}{\bf{SBA}}_n \simeq {\bf 1}^{\binom{n}{0}}\times {\bf 2}^{\binom{n}{1}} \times {\bf 3}_L^{\binom{n}{2}}\times {\bf 4}_L^{\binom{n}{3}} \times\dots \times {\bf{(n+1)}}_L^{\binom{n}{n}}.
$$

that are ${\bf{(m+1)}}_L$ or bigger with the indices $m,m+1,\dots, n$ counting the number of atoms in ${\bf{(m+1)}}_L$, ${\bf{(m+2)}}_L,$ $\dots,$ ${\bf{(n+1)}}_L$.
We have:
\begin{proposition}\label{prop:4.4} ${\bf 2}^{\mu_2} \times {\bf 3}_L^{\mu_3} \times \dots \times {\bf({m+1})}_L^{\mu_{m+1}}$ is a homomorphic image of $_{\mathcal L}{\bf{SBA}}_n$ if and only if
\begin{enumerate}[(i)]
\item $n \geq  m$.
\item $\Gamma_m^n \geq \mu_{m+1}$, $\Gamma_{m-1}^ n \geq \mu_{m+1} + \mu_m$, and in general 
$$
\Gamma_k^n\geq \mu_{m+1}+\mu_m+\dots +\mu_{k+1}
$$
for all indices $k$ where $1 \leq k \leq m$.
\end{enumerate}
Such an $_{\mathcal L}{\bf{SBA}}_n$ exists, with the least such $n$ for which this occurs being the rank $\rho({\bf S})$.

(A similar assertion holds in the rihgt-handed case.)
\end{proposition}

\begin{proof}
This follows from (1)--(3) above plus the fact that, because  ${\bf S}$ is finite, some $_{\mathcal L}{\bf{SBA}}_n$ such that (i) and (ii) hold does exist (it is enough to take $n\geq |{\bf S}|-1$).\end{proof}

For $n=\rho({\bf S})$, $_{\mathcal L}{\bf{SBA}}_n$ is the free cover of ${\bf S}$. Distinct homomorphisms
from $_{\mathcal L}{\bf{SBA}}_n$ onto ${\bf S}$ correspond to distinct ordered generating sets of ${\bf S}$ of size $n$, with each ordered set of generators $a_1, a_2, \dots, a_n$ determining the unique homomorphism that sends $x_i$ to $a_i$ for $i \leq n$. Thus only the free cover $_{\mathcal L}{\bf{SBA}}_n$ need be unique.

Since the partial function algebra ${\bf{\mathcal{P}}}_L(\{1, \dots, n\}, \{1, \dots, m\})$ is isomorphic to ${\bf{(m+1)}}_L^n$, the size of minimal generating sets for ${\bf{(m+1)}}_L^n$ is of interest. Proposition \ref{prop:4.4} leads us to the sequence: $\Gamma_m^m$, $\Gamma_m^{m+1}$ , $\dots$. Here we consider the tails of higher factors of the $_{\mathcal L}{\bf{SBA}}_{m+k}$-sequence, the tails that begin with the ${\bf{(m+1)}}_L$-factor.
$$
{\renewcommand{\arraystretch}{1.7}
\begin{array}{lll}
_{\mathcal L}{\bf{SBA}}_m: &  {\bf{(m+1)}}_L & \Gamma_m^m=1\\
_{\mathcal L}{\bf{SBA}}_{m+1}: & {\bf{(m+1)}}_L^{m+1}\times {\bf{(m+2)}}_L & \Gamma_m^{m+1}=(m+1)+1=m+2\\
_{\mathcal L}{\bf{SBA}}_{m+2}: & {\bf{(m+1)}}_L^{\binom{m+2}{m}}\times {\bf{(m+2)}}_L^{m+2} \times {\bf{(m+3)}}_L& \Gamma_m^{m+2}=\binom{m+2}{m} + (m+2)+1\\
\quad \cdots & \qquad \qquad \cdots & \qquad \quad \cdots \\
_{\mathcal L}{\bf{SBA}}_{m+k}: & {\bf{(m+1)}}_L^{\binom{m+k}{m}}\times \dots \times  {\bf{(m+k+1)}}_L & \Gamma_m^{m+k}=\binom{m+k}{m} + \binom{m+k}{m+1}+\dots + 1
\end{array}
}
$$

As a consequence, ${\bf{(m+1)}}_L^n$ will have a minimal set of $m + k$ generators for all $n$ within the range: $1+\Gamma_m^{m+k-1}\leq n\leq \Gamma_m^{m+k}$.
For ${\bf 3}_L = {\bf{(2+1)}}_L$ and ${\bf 5}_L = {\bf{(4+1)}}_L$ we have:

\vspace{0.2cm}
{\renewcommand{\arraystretch}{1.1}
\begin{center}
\begin{tabular}{ccc}
\begin{tabular}{c|c}
$\rho({\bf 3}_L^n)$ & {Range for $n$ in ${\bf 3}_L^n$} \\
\hline
$2$ & $n=1$\\
\hdashline
$3$ & $2\leq n\leq 4$\\
\hdashline
$4$ & $5\leq n\leq 11$\\
\hdashline
$5$ & $12\leq n\leq 26$\\
\hdashline
$6$ & $27\leq n\leq 57$\\
\end{tabular}
&  \hspace{0.3cm} &
\begin{tabular}{c|c}
$\rho({\bf 5}_L^n)$ & {Range for $n$ in ${\bf 5}_L^n$} \\
\hline
$4$ & $n=1$\\
\hdashline
$5$ & $2\leq n\leq 6$\\
\hdashline
$6$ & $7\leq n\leq 22$\\
\hdashline
$7$ & $23\leq n\leq 64$\\
\hdashline
$8$ & $65\leq n\leq 163$\\
\end{tabular}
\end{tabular}
\end{center}
}

That, for example, $\rho({\bf 3}_L^{57})=6$ is due to the fact that ${\bf 3}_L$ is a homomorphic image of each factor in the tail ${\bf 3}_L^{15} \times {\bf 4}_L^{20} \times {\bf 5}_L^{15} \times {\bf 6}_L^{6} \times {\bf 7}_L^{1}$ of the decomposition of $_{\mathcal L}{\bf{SBA}}_6$. 

The right-handed case is the precise mirror reflection of the above considerations.
Likewise every finite skew Boolean algebra has a free cover, ${\bf{SBA}}_n$, where $n$ is the size of all minimal
generating sets for the algebra considered. If that algebra happens to be, say, left-handed, then its free left-handed cover is the maximal left-handed image of its free two-sided cover.

In the two-sided case, first observe that every primitive factor of ${\bf{SBA}}_n$ has the form ${\bf m}_L\bullet {\bf m}_R$. Thus, taking the given algebra in factored form, first replace each factor ${\bf m}_L\bullet {\bf n}_R$ in it by ${\bf{max}}_L\bullet {\bf{max}}_R$ where ${\mathrm{max}} = \max(m, n)$. For this modified algebra, determine the minimal ${\bf{SBA}}_n$ much as above. This ${\bf{SBA}}_n$ is also minimal for the given algebra.

Returning to ${\bf{(n+1)}}_L^m$, to produce a minimal set of generators for ${\bf{(n + 1)}}_L^m$ is a straightforward, but increasingly tedious process as the parameters $m$ and $n$ increase. We illustrate this process with ${\bf 3}_L^{4}$ using the tail ${\bf 3}_L^{3} \times {\bf 4}_L^{1}$ in the standard decomposition of $_{\mathcal L}{\bf{SBA}}_3$ to first
construct a homomorphism $\varphi$  from the free cover $_{\mathcal L}{\bf{SBA}}_3$ onto ${\bf 3}_L^{4}$. To do so, we describe the atoms of $_{\mathcal L}{\bf{SBA}}_3$ in terms of variable generators, $x$, $y$ and $z$. To begin, $\varphi$ sends all atoms in the atomic classes of $_{\mathcal L}{\bf{SBA}}_3$  prior to those corresponding to ${\bf 3}_L^{3} \times {\bf 4}_L^{1}$ to $0$. On each of these four final factors $\varphi$ is determined by $\varphi(0) = 0$ and:
{\renewcommand{\arraystretch}{1.4}
$$
\begin{array}{ll}
\text{on } \{(x\wedge y)\setminus z, (y\wedge x)\setminus z\} 
& \varphi\colon  (x\wedge y)\setminus z \mapsto 1, (y\wedge x)\setminus z \mapsto 2;\\
\text{on }  \{(x\wedge z)\setminus y, (z\wedge x)\setminus y\} 
& \varphi\colon  (x\wedge z)\setminus y \mapsto 1, (z\wedge x)\setminus y \mapsto 2;\\
\text{on }  \{(y\wedge z)\setminus x, (z\wedge y)\setminus x\} 
& \varphi\colon  (y\wedge z)\setminus x \mapsto 1, (z\wedge y)\setminus x \mapsto 2;\\
\text{on } \{x\wedge y\wedge z, y\wedge z\wedge x, z\wedge x\wedge y\} 
& \varphi\colon  x\wedge y\wedge z \mapsto 1, y\wedge z\wedge x, z\wedge x\wedge y \mapsto 2.
\end{array}
$$
}
Based on the leading variable in each term, $\varphi$ determines the three generators as the images under $\varphi$ of $x$, $y$ and $z$:
$$
x \mapsto (1, 1, 0, 1), \,\, y \mapsto (2, 0, 1, 2), \,\, z \mapsto (0, 2, 2, 2).
$$
Here we have determined the outcomes by seeing, say $x$ for example, as the orthosum of
all atomic terms that have $x$ leading the expression on the left:
$$
x = (x\wedge y)\setminus z + (x\wedge z)\setminus y + x\wedge y\wedge z + \text{other terms sent to } 0.
$$
To confirm that these three $4$-tuples indeed generate ${\bf 3}_L^4$, observe that:
{\renewcommand{\arraystretch}{1.4}
$$
\begin{array}{ll}
x\wedge y\wedge z \mapsto (0, 0, 0, 1), & y\wedge z\wedge x \mapsto (0, 0, 0, 2).\\
(x\wedge y)\setminus z \mapsto (1, 0, 0, 0), &  (y\wedge x)\setminus z \mapsto (2, 0, 0, 0). \\
(x\wedge z)\setminus y \mapsto (0, 1, 0, 0), &  (z\wedge x)\setminus y \mapsto (0, 2, 0, 0). \\
(y\wedge z)\setminus x \mapsto (0, 0, 1, 0), & (z\wedge y)\setminus x \mapsto (0, 0, 2, 0).
\end{array}
$$
}
From these $4$-tuples, all of ${\bf 3}_L^4$ is easily obtained. Returning to partial function algebras, it follows that ${\bf{\mathcal{P}}}_L(\{1, 2,3,4\}, \{1, 2\})$ is minimally generated from the partial maps:
$$
\varphi_1=\left\lbrace\begin{array}{l}1\mapsto 1 \\ 2\mapsto 1\\ 4\mapsto 1\end{array}\right.;  \,\,\, \varphi_2=\left\lbrace\begin{array}{l}1\mapsto 2 \\ 3\mapsto 1\\ 4\mapsto 2\end{array}\right.; \,\,\,
\varphi_3=\left\lbrace\begin{array}{l}2\mapsto 2 \\ 3\mapsto 2\\ 4\mapsto 2\end{array}\right..
$$

\section{Atom Splitting and the Infinite Free Case} \label{s5:atom_splitting}

Consider the inclusion $_{\mathcal L}{\bf{SBA}}_n \subseteq$  $_{\mathcal L}{\bf{SBA}}_{n+1}$ induced by $$\{x_1, \dots, x_n\} \subseteq \{x_1, \dots, x_n, x_{n+1}\}.$$ The atoms of $_{\mathcal L}{\bf{SBA}}_n$ are no longer atomic in  $_{\mathcal L}{\bf{SBA}}_{n+1}$. The left-handed identity $$x = (x\wedge y) + (x\setminus y)$$ gives the following �subatomic� decomposition of the original atoms:
\begin{multline*}
(x_1\wedge x_2\wedge \dots \wedge x_k) \setminus (x_{k+1}\vee \dots \vee x_n)\\
= (x_1\wedge x_2\wedge \dots \wedge x_k \wedge x_{n+1})\setminus (x_{k+1}\vee\dots\vee x_n) + (x_1\wedge x_2\wedge \dots \wedge x_k)\setminus (x_{k+1}\vee \dots \vee x_{n+1}).
\end{multline*}
Both components of the new decomposition are of course atoms in $_{\mathcal L}{\bf{SBA}}_{n+1}$. If say
$(x_1\wedge x_2\wedge x_3)\setminus (x_4\vee \dots \vee x_n) = (x_1\wedge x_3\wedge x_2)\setminus (x_4\vee \dots\vee x_n)$ in $_{\mathcal L}{\bf{SBA}}_n$, then their corresponding pairs
of atomic components in $_{\mathcal L}{\bf{SBA}}_{n+1}$ remain equal. But if say $(x_1\wedge x_2)\setminus (x_3\vee \dots \vee x_n) \neq
(x_2\wedge x_1)\setminus (x_3\vee x_4\vee \dots \vee x_n)$ in $_{\mathcal L}{\bf{SBA}}_n$, then both corresponding pairs of components are
likewise unequal in $_{\mathcal L}{\bf{SBA}}_{n+1}$. One thus has extended the decomposition where, while a given element remains the same, its atomic decomposition doubles in length as each new
generator is added, as in the example of $x\vee y$ following Corollary \ref{cor:3.6}. Thus given $u$ in $_{\mathcal L}{\bf{SBA}}_n$ with atomic decomposition $u = a_1 +\dots + a_r$ in $_{\mathcal L}{\bf{SBA}}_n$, each atom $a_k$ splits as $b_k +
c_k$ in $_{\mathcal L}{\bf{SBA}}_{n+1}$, where $b_i = a_i\wedge x_{n+1}$ and $c_i = a_i\setminus x_{n+1}$, to give a revised atomic decomposition $u = b_1 + c_1 + \dots + b_r + c_r$ in $_{\mathcal L}{\bf{SBA}}_{n+1}$. Given the uniqueness of atomic decompositions (to within commutativity) of elements in $_{\mathcal L}{\bf{SBA}}_n$ or in $_{\mathcal L}{\bf{SBA}}_{n+1}$, we have:

\begin{lemma}\label{lem:5.1}
Given $u$ in $_{\mathcal L}{\bf{SBA}}_n$, let $a$ be an atom of $_{\mathcal L}{\bf{SBA}}_n$ and let $a=b+c$ be the atomic decomposition of $a$ in $_{\mathcal L}{\bf{SBA}}_{n+1}$ where $b = a\wedge x_{n+1}$ and $c = a\setminus x_{n+1}$. Then the
following are equivalent:
\begin{enumerate}[(i)]
\item $u \geq a$ in $_{\mathcal L}{\bf{SBA}}_n$ (and thus in $_{\mathcal L}{\bf{SBA}}_{n+1}$).
\item $u \geq b$ in $_{\mathcal L}{\bf{SBA}}_{n+1}$.
\item $u\geq c$ in $_{\mathcal L}{\bf{SBA}}_{n+1}$.
\end{enumerate}
\end{lemma}

Thus $a$ is in the atomic decomposition of $u$ in $_{\mathcal L}{\bf{SBA}}_n$ iff $b$ (or $c$ and hence both) is in the atomic decomposition of $u$ in $_{\mathcal L}{\bf{SBA}}_{n+1}$.


This leads us to infinite free algebras with necessarily infinite generating sets. If
$X$ is infinite, then $_{\mathcal L}{\bf{SBA}}_X$ is the upward directed union of its finite free subalgebras:
$$
_{\mathcal L}{\bf{SBA}}_X =\bigcup \{_{\mathcal L}{\bf{SBA}}_Y\colon \varnothing \neq Y\subseteq X \text{ and } |Y|<\infty\}.
$$

Given $u$ and $v$ of $_{\mathcal L}{\bf{SBA}}_X$, each occurs in some finite free subalgebra, say $u$ in $_{\mathcal L}{\bf{SBA}}_Y$ and
$v$ in $_{\mathcal L}{\bf{SBA}}_Z$ for finite subsets $Y$ and $Z$ of $X$. Thus $u\wedge v$, $u\vee v$ and $u\setminus v$ are calculated in the
larger finite subalgebra $_{\mathcal L}{\bf{SBA}}_{Y\cup Z}$ or in any finite $_{\mathcal L}{\bf{SBA}}_W$ where $Y\cup Z \subseteq W$. Of course,
calculations of $u\wedge v$, $u\vee v$ and $u\setminus v$ do not change in passing from $_{\mathcal L}{\bf{SBA}}_{Y\cup Z}$ to any properly
larger $_{\mathcal L}{\bf{SBA}}_W$. What changes is their atomic decompositions; such changes, however,
are derived from the original decompositions in $_{\mathcal L}{\bf{SBA}}_{Y\cup Z}$ by (possibly repeated) atomic
splitting. Ultimately, we obtain:

\begin{proposition} In $_{\mathcal L}{\bf{SBA}}_X$ for $X$ infinite, no atoms exist. 
\end{proposition}
\begin{proof}
If $a$ is an atom, then it
appears as such in $_{\mathcal L}{\bf{SBA}}_Y$ for some finite $Y$; but it immediately looses its atomic status in a properly larger free subalgebra.
\end{proof}

This is a fundamental difference between finite and infinite free algebras. It leads, in turn, to the second difference.
Recall that the {\em center} of a skew lattice, consisting of elements that both $\wedge$- commute and $\vee$-commute with all elements, is the union of all singleton ${\mathcal{D}}$-classes. (\cite{L3}
Theorem 1.7.) In $_{\mathcal L}{\bf{SBA}}_n$ (or $_{\mathcal R}{\bf{SBA}}_n$ or ${\bf{SBA}}_n$) it is the set of all $n$ atoms of the form $x_1\setminus (x_2\vee\dots\vee x_n)$ and the subalgebra they generate consisting of all orthosums of such atoms. But, except for $0$, none of these orthosums remain central in $_{\mathcal L}{\bf{SBA}}_{n+1}$. For each atom, we can write
\begin{align*}
x_{n+1}\wedge (x_1\setminus (x_2\vee \dots \vee x_n)) & = (x_{n+1}\wedge x_1)\setminus (x_2\vee\dots\vee x_n)\\
& \neq (x_1\wedge x_{n+1})\setminus (x_2\vee\dots\vee x_n) \\
& = (x_1\setminus (x_2\vee\dots\vee x_n))\wedge  x_{n+1},
\end{align*}
with the two new, unequal atoms, being ${\mathcal{D}}$-related in $_{\mathcal L}{\bf{SBA}}_{n+1}$. Thus, given any non-zero central element $c = a_1 + \dots + a_k$ in $_{\mathcal L}{\bf{SBA}}_n$ with atoms $a_i$ of the given form,
$$x_{n+1}\wedge c = (x_{n+1}\wedge a_1) + \dots  + (x_{n+1}\wedge a_k) \neq (a_1\wedge x_{n+1}) + \dots + (a_k\wedge x_{n+1}) = c\wedge x_{n+1}.$$ The case for $_{\mathcal R}{\bf{SBA}}_n$ and $_{\mathcal R}{\bf{SBA}}_{n+1}$, or ${\bf{SBA}}_n$ and ${\bf{SBA}}_{n+1}$, is similar. We thus have:

\begin{theorem} \label{th:5.2} Given a finite free skew Boolean algebra on $n$ generators, whether left-handed, right-handed or two-sided, its center forms a Boolean algebra of order $2^n$. In the case of an infinite free algebra, the center is just $\{0\}$.
\end{theorem}


Recall that a skew lattice  $(S; \wedge, \vee)$ has {\em intersections} (or is {\em with intersections}) if every pair $e, f \in S$ possesses a {\em natural} meet with respect to the natural partial order $\geq$ on S. If it exists, the natural meet of $e$ and $f$ is denoted by by $e\cap f$ and called the {\em intersection} of $e$ and $f$. For any pair $e$ and $f$, $e\wedge f$ coincides with $e\cap f$ if and only if $e\wedge f = f\wedge e$. By \cite[Theorem 2.8]{BL}, skew Boolean algebras with intersections form a congruence distributive variety. Returning to Example \ref{ex:1},  given partial functions $f$ and $g$ in ${\mathcal{P}}(A, B)$, $f \cap g$ is precisely their intersection when viewed as subsets of $A\times B$. 


If we can show that $_{\mathcal L}{\bf{SBA}}_X$ has intersections for all sets $X$, then so does the term equivalent $_{\mathcal R}{\bf{SBA}}_X$, and we will see also ${\bf{SBA}}_X$. This is all clear for $X$ finite since finite
SBAs have intersections. Indeed, given a finite product ${\bf P}_1 \times \dots \times {\bf P}_r$ of primitive SBAs (themselves finite or not) the intersection exists and is given by:
$$
(p_1,\dots, p_r)\cap (q_1,\dots, q_r)=(p_1\cap q_1,\dots, p_r\cap q_r) \text{ where each } p_k\cap q_k=\left\lbrace\begin{array}{ll}p_k, & \text{if }p_k=q_k;\\
0,& \text{otherwise.}\end{array}\right.
$$
 Internally viewed, in any finite SBA, $x\cap y$ is the orthosum of all atoms common to the
atomic decompositions of both $x$ and $y$, which must be $0$ when no such atoms exist. Thanks to our observations on adjoining free generators and the effects on atoms,
including Lemma~\ref{lem:5.1}, intersections are stable under the inclusion $_{\mathcal L}{\bf{SBA}}_n \subset$ $_{\mathcal L}{\bf{SBA}}_{n+1}$. Thus the intersection for elements in $_{\mathcal L}{\bf{SBA}}_3$ remains the same for these elements in the bigger, say $_{\mathcal L}{\bf{SBA}}_7$. What changes is the decomposition of all outcomes into atoms. The pool of atoms that two elements share in $_{\mathcal L}{\bf{SBA}}_n$, doubles by splitting to give rise to the new pool of atoms in $_{\mathcal L}{\bf{SBA}}_{n+1}$ that both share. As a result we obtain:

\begin{theorem}\label{th:5.3} Given any set X, the free left-handed (right-handed) skew Boolean algebra $_{\mathcal L}{\bf{SBA}}_X$ ($_{\mathcal R}{\bf{SBA}}_X$) on $X$ has intersections. Given elements $x$ and $y$ in $_{\mathcal L}{\bf{SBA}}_X$, $x\cap y$
can be calculated in any subalgebra $_{\mathcal L}{\bf{SBA}}_Y$, where $Y$ is any finite subset of $X$ such that $_{\mathcal L}{\bf{SBA}}_Y$ contains both $x$ and $y$. Similar remarks hold for $_{\mathcal R}{\bf{SBA}}_X$.
\end{theorem}

\begin{proof} Suppose that $x$ and $y$ are encountered in $_{\mathcal L}{\bf{SBA}}_Y$ where $Y$ is a finite subset of $X$ and that $u$ in $_{\mathcal L}{\bf{SBA}}_X$ is such that $u$  is less than or equal to both $x$ and $y$. Then $x\cap y$ relative to $_{\mathcal L}{\bf{SBA}}_Y$ exists. By our remarks, this $x\cap y$ remains the intersection in any $_{\mathcal L}{\bf{SBA}}_Z$ where $Y \subseteq Z$ if $Z$ is finite. Now $u$ must be encountered in some finite subalgebra $_{\mathcal L}{\bf{SBA}}_U$ where $U \subseteq X$. Then both the
current $x\cap y$ and $u$ must lie in the larger subalgebra $_{\mathcal L}{\bf{SBA}}_{Y\cup U}$. Since $Y\cup U$ is finite, $x\cap y$ is the intersection here also, and $u \leq x\cap y$ follows. Thus $x\cap y$ remains the intersection of $x$ and $y$ throughout all of $_{\mathcal L}{\bf{SBA}}_X$. The case for $_{\mathcal R}{\bf{SBA}}_Y$ is similar.
\end{proof}

Since $_{\mathcal R}{\bf{SBA}}_X\simeq  {\bf{SBA}}_X/{\mathcal L}$, $_{\mathcal L}{\bf{SBA}}_X\simeq  {\bf{SBA}}_X/{\mathcal R}$ and by \cite[Theorem 2.3]{LS} a skew Boolean algebra ${\bf S}$ has intersections if and only if both ${\bf S}/{\mathcal L}$ and ${\bf S}/{\mathcal R}$ have them, we thus
have:

\begin{corollary}\label{cor:5.4}
All free skew Boolean algebras have intersections.
\end{corollary}

Returning to surjective homomorphisms, any such map $f$ from a finite SBA ${\bf S}$ preserves intersections if and only if it is determined by its kernel ideal, $\{x \in S\colon f(x) = 0\}$ (See \cite[Proposition 3.8, Theorem 3.9]{BL}. See also the descriptions of $\cap$-morphisms between SBAs with intersections given in \cite{BCV} and \cite{Kud}.)
For example, in the case
of ${\bf 2}^{\mu_2} \times {\bf 3}_L^{\mu_3} \times \dots \times {\bf{m}}_L^{\mu_{m}}$, the kernel ideal is a sub-product of primitive factors so that the
map, to within isomorphism, is just a projection onto the product of the complementary
factors. Thus a $\cap$-preserving homomorphism from ${\bf 2}^{\mu_2} \times {\bf 3}_L^{\mu_3} \times \dots \times {\bf{m}}_L^{\mu_{m}}$ is essentially a
projection onto a direct factor, ${\bf 2}^{\nu_2} \times {\bf 3}_L^{\nu_3} \times \dots \times {\bf{m}}_L^{\nu_{m}}$, with all $\nu_j \leq \mu_j$ and possibly some $\nu_i = 0$.

In general, an algebra ${\bf A}$ is free on a subset $X$ of the underlying set A of ${\bf A}$, if for any other algebra ${\bf B}$, every map $f\colon X \to B$ extends uniquely to a homomorphism $f\colon {\bf A}\to {\bf B}$. When this is the case, $X$ must generate ${\bf A}$ and we say that $X$ {\em freely generates} ${\bf A}$. It is easily seen that the identity map on $X$ induces an inverse pair of isomorphisms between ${\bf F}_X$ and  ${\bf A}$, where ${\bf F}_X$ denotes the relevant term algebra, which in our case is ${\bf{SBA}}_X$,
$_{\mathcal R}{\bf{SBA}}_X$, $_{\mathcal L}{\bf{SBA}}_X$ or ${\bf{GBA}}_X$.

\begin{theorem}\label{th:5.5} Let ${\bf S}$ be a left-handed (right-handed) skew Boolean algebra, let $X\subseteq S$ be a generating set of ${\bf S}$ and let $\pi\colon{\bf S} \to {\bf S}/{\mathcal D}$ be the canonical homomorphism. Then the following statements are equivalent:
\begin{enumerate}[(i)]
\item ${\bf S}$ is freely generated by $X$.
\item For every finite subset $Y$ of $X$, the subalgebra $\langle Y\rangle$ generated by $Y$ is free on $Y$.
\item For every subset $\{x_1, \dots, x_n\}$ of $n$ distinct elements in $X$, their evaluations in the $n2^{n-1}$ atomic terms on $n$ variables produce $n2^{n-1}$ distinct non-zero outcomes in ${\bf S}$.
\item ${\bf S}/{\mathcal D}$ is freely generated by $\pi(X)$ and for any $x\neq y$ in $X$ and any $a\in S$ if $a\leq x,y$, then $a = 0$.
\item ${\bf S}/{\mathcal D}$ is freely generated by $\pi(X)$ and for any $x \neq y \in X$, $x\cap y$ exists and equals $0$.
\end{enumerate}
\end{theorem}
\begin{proof}
First observe that (ii) and (iii) are equivalent by Proposition \ref{prop:4.1}, while (iv) and  (v) are equivalent by the meaning of $x\cap y$. By remarks preceding the theorem, since  (ii) holds for $_{\mathcal L}{\bf{SBA}}_X$,  (ii) and  (iii) follow from (i). Item (i) implies  (iv) and (v) is since distinct free generators share no common atoms in finite free algebras (check $x\cap y$ in $_{\mathcal L}{\bf{SBA}}_{\{x,y\}}$) and thus have $0$-intersection in any free algebra. For the converse implications, first assume  (ii) and let $f\colon _{\mathcal L}{\bf{SBA}}_X \to {\bf S}$ be the homomorphism induced by the map $x\mapsto x$ $f$ is clearly
surjective. Is $f$ one-to-one? If not then on some finite free subalgebra $_{\mathcal L}{\bf{SBA}}_Y$ for $Y$ a finite subset of $X$, $f$ will not be one-to-one. But (ii) prevents this so that $f$ is an
isomorphism and (i) follows. We now have  (i) - (iii) equivalent. Finally assume  (iv) and (v). We show (iii) by first observing that the assumption on ${\bf S}/{\mathcal D}$ guarantees that $\langle Y\rangle$ has
the correct number of atomic classes for any finite $Y \subseteq X$. The assumption that $x\cap y = 0$ for all $x \neq y$ in $X$ ensures that for any such $Y$, the atomic classes in $\langle Y\rangle$ have the size given in the free case. Indeed  $x\cap y = 0$ for some $x \neq y$ in $Y$ implies that all pairs of atoms of the form $(x\wedge y\wedge u)\setminus v$ and $(y\wedge x\wedge u)\setminus v$, while ${\mathcal L}$-related, must be distinct. Further, since this is so
for all $x \neq y$ in $Y$, all atomic classes of $\langle Y\rangle$ obtain the maximal size allowed in $_{\mathcal L}{\bf{SBA}}_Y$, making $\langle Y\rangle \simeq$  $_{\mathcal L}{\bf{SBA}}_Y$. Items (i) - (iii)  now follow.
\end{proof}

\begin{corollary} For $n\geq 2$, the free (right-handed, left-handed) skew Boolean algebra on $n$ generators is not free as a (right-handed, left-handed) skew Boolean algebra with intersections.
\end{corollary}
\begin{proof} Given distinct generators $x \neq y$, while $x\cap y = 0$ in a free SBA algebra, this is not the case when $\cap$ is brought into the signature, since $x\cap y = 0$ is not an identity for skew Boolean algebras with intersections.
\end{proof}

We apply Theorem \ref{th:5.5} to give an alternative construction, suggested by~\cite{Kud},  of a free left-handed
algebra on any non-empty set $X$.   We begin with the set $\{0, 1\}^X$ of all maps $X\to \{0, 1\}$.
Effectively we just work with ${\mathcal X} = \{0, 1\}^X \setminus \{f_0\}$ where $f_0$ is the zero map.  Next, we set
$$
\Omega = \{(f, x)\colon f \in {\mathcal X} \text{ and } x \in X \text{ is such that } f(x) = 1\}.
$$
and define $p\colon \Omega\to {\mathcal X}$ by $p(f,x)=f$. If ${\bf S}_{\Omega}$ is the class of all subsets of $\Omega$, then on ${\bf S}_{\Omega}$ we define the binary operations  $\vee$, $\wedge$ and $\setminus$ by:
{\renewcommand{\arraystretch}{1.4}
$$
\begin{array}{lcl}
A \wedge B & = & \{(f, x) \in A\colon f \in  p(A) \cap p(B)\},\\
A\vee B & = & (A\setminus B)\cup B = \{(f,x)\in A\cup B\colon f\in p(A)\setminus p(B) \text{ or } f\in p(B)\},\\
A \setminus  B & = &\{(f, x) \in A\colon f \in p(A)\setminus  p(B)\}.
\end{array}
$$
}
The following statement is easily verified:

\begin{proposition}\label{prop:5.7}
$({\bf S}_{\Omega};  \wedge,\vee \setminus, \varnothing)$ is a left-handed skew Boolean algebra.
\end{proposition}

Define $i\colon X \to {\bf S}_{\Omega}$ by $i(x) = \{(f, x)\colon f(x) = 1\}$. This map is clearly injective. We next let
 $\overline{X} = \{i(x) \colon x \in X\}$ and let ${\bf S}_X = \langle \overline{X}\rangle$ be the subalgebra of ${\bf S}_{\Omega}$ generated by $\overline{X}$.

\begin{theorem}\label{th:5.8}
${\bf S}_X$ is freely generated by $\overline{X}$.
\end{theorem}
\begin{proof} Given a finite subset $\{x_1, x_2, \dots, x_n\}$ of $X$, we show that each atomic term when evaluated on $\{i(x_1), i(x_2), \dots, i(x_n)\}$ is non-empty with distinct atomic terms having different evaluations. Observe that:
{\renewcommand{\arraystretch}{1.1}
$$
\begin{array}{lcl}
i(x)\wedge i(y) & = & \{(f, x) \in \Omega\colon f(x) = f(y) = 1\} = \{(f, x) \in \Omega\colon f(y) = 1\},\\
i(x) \setminus i(y)& = & \{(f, x) \in \Omega\colon f(x) = 1 \text{ and } f(y) = 0\} = \{(f, x) \in \Omega\colon f(y) = 0\}.
\end{array}
$$
}
As a consequence we obtain:
\begin{multline*}
(i(x)\wedge i(y) \wedge \dots \wedge i(z)) \setminus (i(u)\vee  \dots \vee i(w))\\
= \{(f, x) \in \Omega \colon f(x) = f(y) =\dots  = f(z) = 1 \text{ and } f(u) =\dots = f(w) = 0\}.
\end{multline*}
As long as the elements $x$ chosen from $X$ are distinct, none of the evaluated terms will be empty; and distinct ${\mathcal L}$-related atomic terms must yield distinct outcomes. Hence the conditions of  Theorem \ref{th:5.5}(iii) are satisfied.
\end{proof}

We finally provide an explicit connection of the construction of ${\bf S}_X$ with the duality theory of \cite{Kud}. The triple $(\Omega,p,{\mathcal X})$ can be shown to be (homeomorphic to) the dual \'etale space of $_{\mathcal L}{\bf{SBA}}_X$. Thus ${\bf S}_X$ is an isomorphic copy of $_{\mathcal L}{\bf{SBA}}_X$, which arises as the left-handed skew Boolean algebra of all compact-open sections, with respect to the operations of left restriction, right override and difference, of the dual \'etale space of $_{\mathcal L}{\bf{SBA}}_X$.


\section*{Acknowledgement} The referee's comments have led to improvements in the paper's presentation.


\end{document}